\documentclass[11pt]{amsart}

\usepackage{amssymb}
\usepackage{bm}
\usepackage[centertags]{amsmath}
\usepackage{amsfonts}
\usepackage{amsthm}
\usepackage{mathrsfs}
\usepackage[usenames]{xcolor}
\usepackage[normalem]{ulem}
\usepackage{comment}
\usepackage{enumitem}
\makeatletter
\def\l@subsection{\@tocline{2}{0pt}{2.5pc}{2.5pc}{}}%
\makeatother
\usepackage{hyperref}
\hypersetup{pdftex,colorlinks=true,allcolors=black}
\usepackage{hypcap}

\DeclareFontFamily{OT1}{pzc}{}
\DeclareFontShape{OT1}{pzc}{m}{it}{<-> s * [1.10] pzcmi7t}{}
\DeclareMathAlphabet{\mathpzc}{OT1}{pzc}{m}{it}

\theoremstyle{definition}
\newtheorem{thm}{Theorem}[section]
\newtheorem{dfn}[thm]{Definition}
\newtheorem{lem}[thm]{Lemma}
\newtheorem{prop}[thm]{Proposition}
\newtheorem{cor}[thm]{Corollary}
\newtheorem{rem}[thm]{Remark}
\newtheorem{exa}[thm]{Example}
\newtheorem{ass}[thm]{Assumption}

\newtheorem{inst}[thm]{Instruction}

\newtheorem*{defn*}{Definition}

\newtheorem*{thm*}{Theorem}

\newtheorem*{cor*}{Corollary}

\newtheorem*{prp*}{Proposition}

\newtheorem{problem}{Problem}

\newtheorem{thmA}{Theorem}

\newcommand{\N}{\mathbb{N}}
\newcommand{\Z}{\mathbb{Z}}

\newcommand{\Q}{\mathbb{Q}}
\newcommand{\C}{\mathbb{C}}
\newcommand{\D}{\mathbb{D}}
\newcommand{\e}{\varepsilon}
\newcommand{\ga}{\gamma}

\newcommand{\De}{\Delta}

\newcommand{\la}{\lambda}

\newcommand{\Ga}{\Gamma}

\newcommand{\X}{\mathfrak{X}_{^{C\>\!\!(\>\!\!K\>\!\!)}}}

\newcommand{\trn}[1]{{\left\vert\kern-0.25ex\left\vert\kern-0.25ex\left\vert #1 
    \right\vert\kern-0.25ex\right\vert\kern-0.25ex\right\vert}}

\newcommand{\trnsmall}[1]{{\vert\kern-0.25ex\vert\kern-0.25ex\vert #1 
    \vert\kern-0.25ex\vert\kern-0.25ex\vert}}

\DeclareMathOperator{\supp}{supp}
\DeclareMathOperator{\range}{range}
\DeclareMathOperator{\rank}{rank}
\DeclareMathOperator{\weight}{weight}
\DeclareMathOperator{\age}{age}

\long\def\symbolfootnote[#1]#2{\begingroup%
\def\thefootnote{\fnsymbol{footnote}}\footnote[#1]{#2}\endgroup}

\makeatletter
\@namedef{subjclassname@2020}{%
  \textup{2020} Mathematics Subject Classification}
\makeatother
\allowdisplaybreaks

\begin{document}

\title[Separable $C(K)$ Spaces as Calkin Algebras]{Separable Spaces of Continuous Functions as Calkin Algebras}
%\dedicatory{}

\author[P. Motakis]{Pavlos Motakis}

\address{P. Motakis, Department of Mathematics and Statistics, York University, 4700 Keele Street, Toronto, Ontario, M3J 1P3, Canada}

\email{pmotakis@yorku.ca}

\thanks{The author was supported by NSERC Grant RGPIN-2021-03639.}

%\date{\today}

\keywords{}

\subjclass[2020]{46B07, 46B25, 46B28, 46J10.}

\begin{abstract}
It is proved that for every compact metric space $K$ there exists a Banach space $X$ whose Calkin algebra $\mathcal{L}(X)/\mathcal{K}(X)$ is homomorphically isometric to $C(K)$. This is achieved by appropriately modifying the Bourgain-Delbaen $\mathscr{L}_\infty$-space of Argyros and Haydon in such a manner that sufficiently many diagonal operators on this space are bounded.
\end{abstract}

\maketitle

%-----------To omit subsections in table of contents
%\setcounter{tocdepth}{1}

\tableofcontents

\section{Introduction}
For a Banach space $X$ denote by $\mathcal{L}(X)$ the algebra of bounded linear operators on $X$ and by $\mathcal{K}(X)$ the compact operator ideal in $\mathcal{L}(X)$. A Banach algebra $\mathcal{A}$ is said to be a Calkin algebra if there exists an underlying Banach space $X$ so that the Calkin algebra $\mathpzc{Cal}(X) = \mathcal{L}(X)/\mathcal{K}(X)$ of $X$ is isomorphic as a Banach algebra (not necessarily isometrically) to $\mathcal{A}$. The question of what unital Banach algebras are Calkin algebras is very rudimentary. Calkin introduced this object for $X = \ell_2$ in 1941 (\cite{calkin:1941}). Since then $\mathpzc{Cal}(\ell_2)$ has uninterruptedly been in the spotlight, partly owing to the fact that it has highlighted connections between operator algebras and other fields of mathematics, e.g., $K$-theory (Brown, Douglas, and Fillmore, \cite{brown:dougla:fillmore:1977}), Set Theory (Phillips and Weaver, \cite{phillips:weaver:2007}), and Descriptive Set Theory (Farah, \cite{farah:2011}). The systematic study of $\mathpzc{Cal}(X)$ for general Banach spaces $X$ dates back to Yood's 1954 paper \cite{yood:1954}. The term Calkin algebra in this precise context can be traced at least as far back as 1974 to Caradus, Pfaffenberger, and Yood's  book \cite{caradus:pfaffenberger:yood:1974} who proposed the problem of specifying criteria on $X$ which would determine whether $\mathpzc{Cal}(X)$ is semi-simple. The advent of powerful modern construction techniques in Banach spaces, such as the Gowers-Maurey (\cite{gowers:maurey:1993}) and Argyros-Haydon (\cite{argyros:haydon:2011}) methods, made it finally possible to represent certain relatively simple Banach algebras as Calkin algebras. Notable examples include the complex field (\cite{argyros:haydon:2011}), the semigroup algebra of $\N_0$ (Tarbard, \cite{tarbard:2013}), and $C(K)$ for a countable compact space $K$ (Puglisi, Zisimopoulou, and the author, \cite{motakis:puglisi:zisimopoulou:2016}). Although this is an impressive fact, more complicated Banach algebras, such as $C[0,1]$, have been entirely out of reach with past methods. The purpose of the current paper is to develop a new technique that bridges this gap. To overcome existing limitations, a radically new way of imposing a prescribed structure of operators on a Banach space $X$ is presented. As a result, for every compact metric space $K$, $C(K)$ admits a representation as a Calkin algebra.

Within the context of modern Banach space theory, there is a strong relation between the explicit description of quotient algebras of $\mathcal{L}(X)$ and the tight control of the structure of bounded linear operators on a Banach space $X$. The most relevant examples to the present paper are the Gowers-Maurey space $X_\mathrm{GM}$ from 1993 (\cite{gowers:maurey:1993}) and the Argyros-Haydon space $\mathfrak{X}_\mathrm{AH}$ from 2011 (\cite{argyros:haydon:2011}). Each of these groundbreaking constructions solved numerous longstanding open problems and revolutionized the view of general Banach spaces. Both spaces have the tightest possible space of operators modulo a small ideal, which is different in each case. An operator between Banach spaces is called strictly singular if it does not preserve an isomorphic copy of any infinite dimensional subspace of its domain. Denote by $\mathcal{SS}(X)$ the ideal of strictly singular operators on $X$. Then, every $T\in\mathcal{L}(X_\mathrm{GM})$ is a scalar operator plus a strictly singular operator and every $T\in\mathcal{L}(\mathfrak{X}_\mathrm{AH})$ is a scalar operator plus a compact one. In other words, $\mathcal{L}(X_\mathrm{GM})/\mathcal{SS}(X_\mathrm{GM})$ and $\mathcal{L}(\mathfrak{X}_\mathrm{AH})/\mathcal{K}(\mathfrak{X}_\mathrm{AH})$ are both one-dimensional. The space $\mathfrak{X}_\mathrm{AH}$ was constructed by combining the Bourgain-Delbaen method for defining $\mathscr{L}_\infty$-spaces from \cite{bourgain:delbaen:1980} with the Gowers-Maurey space $X_\mathrm{GM}$. It is now understood that frequently phenomena that can be witnessed ``modulo strictly singular operators'' in Gowers-Maurey-type spaces can also be witnessed ``modulo compact operators'' in Argyros-Haydon-type spaces (see, e.g., \cite{tarbard:2013}, \cite{manoussakis:pelczar:swietek:2017},\cite{argyros:motakis:2019}).

Gowers and Maurey sought  in \cite{gowers:maurey:1997} the introduction of a prescribed structure on the space of operators $\mathcal{L}(X)$ of a Banach space $X$. Building upon their aforementioned paper \cite{gowers:maurey:1993} they developed a general method for defining Banach spaces whose algebra of operators has a quotient algebra that is generated by what they called a proper family of spread operators. Among other examples, they defined a space $X_S$ with a basis on which both the left shift $L$ and the right shift $R$ are bounded. In fact, any operator $T\in\mathcal{L}(X_S)$ can be written as $T = \la I + \sum_{n=1}^\infty a_n R^n + \sum_{n=1}^\infty b_nL^n + S$, where $S\in\mathcal{SS}(X_S)$ and the scalar coefficients in both series are absolutely summable. As a consequence, the quotient algebra $\mathcal{L}(X_S)/\mathcal{SS}(X_S)$ coincides with the convolution algebra $\ell_1(\Z)$ (also known as the Wiener algebra).

Tarbard adapted some of the techniques from \cite{gowers:maurey:1997} and defined an Argyros-Haydon-type $\mathscr{L}_\infty$-space $\mathfrak{X}_\mathrm{T}$ in \cite{tarbard:2013} on which a kind of right shift operator $R$ is bounded. On this space every bounded linear poperator $T$  can be writen as $T = \la I + \sum_{n=1}^\infty a_n R^n + A$ with $A\in\mathcal{K}(\mathfrak{X}_\mathrm{T})$ and the scalar coefficients in the series being absolutely summable. Thusly, the resulting Calkin algebra of $\mathfrak{X}_\mathrm{T}$ is the convolution algebra $\ell_1(\N_0)$. With similar techniques, he had earlier constructed in \cite{tarbard:2012} for each $n\in\N$ a space whose Calkin algebra is the algebra of $n\times n$ upper triangular Toeplitz matrices.

There is a natural three-step process for defining a space with a prescribed Calkin algebra $\mathcal{A}$. The first step is to identify an appropriate class $\mathcal{C}$ of operators, acting on a classical Banach space $X_0$ with a basis, that generates $\mathcal{A}$ in $\mathcal{L}(X_0)$. For example, the left and right shift acting on $X_0=\ell_1(\Z)$ generate the Wiener algebra. The next step is to adapt the methods from \cite{gowers:maurey:1997} to define a Gowers-Maurey-type space $X$ on which this predetermined class of operators can be used to approximate all $T$ in $\mathcal{L}(X)$, modulo the strictly singular operators. The final step is to involve the Bourgain-Delbaen construction method to produce an Argyros-Haydon-type space whose Calkin algebra is explicitly $\mathcal{A}$. How well this process works depends on the class $\mathcal{C}$ and the space $X_0$. Classes of spread operators on $\ell_1$ have fitted within this framework nicely. Without elaborating too much on the reason of this success, the second step comes down to the fact that $X_\mathrm{GM}$, being based on Schlumprecht space (\cite{schlumprecht:1991}), has a lot of local $\ell_1$-structure and that shift operators don't interfere too much with conditional structure. For the third step, it is important that bounded shift operators can already be found (\cite[Theorem 3.7, page 742]{tarbard:2012}) in a certain ``simple'' mixed-Tsirelson Bourgain-Delbaen $\mathscr{L}_\infty$-space of Argyros and Haydon (\cite[Section 4]{argyros:haydon:2011}). The present paper follows this paradigm but, due to inherent limitations of all previous Bourgain-Delbaen constructions, it cannot be carried out without drastic conceptual modification.

There are also examples of explicit Calkin algebras that do not adhere to this specific three-step process. For example, in \cite{kania:laustsen:2017} it was observed by Kania and Laustsen that by combining finitely many carefully chosen Argyros-Haydon spaces one can obtain any finite dimensional semi-simple complex algebra as a Calkin algebra. An earlier, but more involved, instance of this type of construction is due to Puglisi, Zisimopoulou, and the author from \cite{motakis:puglisi:zisimopoulou:2016}. The main idea is that by taking infinitely many Argyros-Haydon spaces and combining them with an Argyros-Haydon sum (introduced by Zisimopoulou in \cite{zisimopoulou:2014}) the resulting space has Calkin algebra $C(\omega)$. Iterating this process yields, for every countable compactum $K$, a $C(K)$ Calkin algebra. A similar method was used by Puglisi, Tolias, and the author  in \cite{motakis:puglisi:tolias:2020} to construct a variety of Calkin algebras, e.g., quasi-reflexive and hereditarily indecomposable ones. All underlying Banach spaces mentioned in this paragraph may be viewed as composite Argyros-Haydon spaces.

Although the statement of the main result of this paper is very similar to the aforementioned one from \cite{motakis:puglisi:zisimopoulou:2016} the proof is very different. As it was pointed out in that paper, the iterative method is insufficient for uncountable $K$. Here, the concept of introducing a prescribed structure of operators in a Bourgain-Delbaen space is built from the ground up to achieve the desired result. With regards to the first step of the three-step process towards constructing a $C(K)$ Calkin algebra, normal operators on $\ell_2$ are the natural candidates. Indeed, by the spectral theorem the $C^*$-algebra generated by a normal operator $T$ is $C(\sigma(T))$. This is in fact fairly straightforward whenever $T$ is diagonal with respect to some orthonormal basis and thus its norm is the supremum of its diagonal entries. Going through this exercise clarifies that on any space with an unconditional basis any family of diagonal operators generates a $C(K)$ Banach algebra. This supports the conclusion that the class $\mathcal{C}$ in the first step needs to be one consisting of sufficiently many diagonal operators that capture the information of the compact space $K$.

\begin{thmA}
\label{main theorem}
Let $K$ be a compact metric space. There exists an Argyros-Haydon-type Bourgain-Delbaen $\mathscr{L}_\infty$-space $\X$ with a conditional Schauder basis $(d_\ga)_{\ga\in\Ga}$ that satisfies the following properties.
\begin{enumerate}[leftmargin=19pt,label=(\alph*)]

\item\label{main diag+comp} Every bounded linear operator $T:\X\to\X$ can be written in the form $T = D+A$ where $D$ is diagonal bounded linear operator and $A$ is a compact linear operator.

\item\label{main identification} There exists a Banach algebra isomorphism $\Psi:C(K)\to\mathpzc{Cal}(\X)$.

\item\label{main isometry} The space $\X$ admits an equivalent norm with respect to which $\Psi$ is an isometry. In particular, this space's Calkin algebra is homomorphically isometric to $C(K)$.

\end{enumerate}
\end{thmA}

Statement \ref{main isometry} is separate to indicate the fact that the usual Bourgain-Delbaen norm does not have the isometric property.

As it has been mentioned repeatedly, the construction of $\X$ cannot be an immediate application of the aforementioned three-step process. Although the first and second step would go through with reasonably standard modifications of classical methods, the third one meets a dead end; on all types of previously defined Bourgain-Delbaen spaces non-trivial diagonal operators are never bounded. To overcome this, it is necessary to model the space $\X$ on a different Gowers-Maurey-type space that relies heavily on a technique called saturation under constraints initiated by Odell and Schlumprecht (\cite{odell:schlumprecht:1995} and \cite{odell:schlumprecht:2000}) and extensively developed by Argyros, the author, and others (\cite{argyros:motakis:2014}, \cite{argyros:motakis:2019}, etc.). In the end, boundedness of non-scalar operators is achieved in a fundamentally different way compared to previous non-classical spaces (e.g., \cite{gowers:1994}, \cite{gowers:maurey:1997}). Significant effort has been made to explain this new idea. To this end, Section \ref{mT section} is entirely devoted to an exposition of concepts in a less intimidating mixed-Tsirelson stage. None of the results from that section are used directly in the proof of Theorem \ref{main theorem}. Instead, everything needs to be reframed in the setting of the Argyros-Haydon construction. However, the main point is that the final result is based on an accessible novel principle that is then thrusted by powerful existing technologies to fulfill its potential.

The paper is organized into nine sections. Arguably, the most important one is Section \ref{mT section} in which the underlying principles behind Theorem \ref{main theorem} are clarified. This is done by defining and briefly studying a mixed-Tsirelson space $X_{^{C\>\!\!(\>\!\!K\>\!\!)}}$. Section \ref{BD section} deals with the development of the Argyros-Haydon-type Bourgain-Delbaen incarnation $\X$ of the simpler space $X_{^{C\>\!\!(\>\!\!K\>\!\!)}}$. In Section \ref{diagonal boundedness section} it is shown that sufficiently many diagonal operators on $\X$ are bounded to define a homomorphic embedding $\Psi:C(K)\to\mathpzc{Cal}(\X)$. In Section \ref{impact section} the main Theorem is proved by showing that $\Psi$ is onto. This argument is made modulo two black-box theorems that rely on the conditional structure of $\X$, which is studied in Sections \ref{common concepts section} and \ref{operators section}. In these two sections, techniques from the theory of hereditarily indecomposable (HI) spaces are implemented and a satisfactory first reading of the paper is possible by omitting them. Some fundamental and non-trivial parts of the Argyros-Haydon construction, such as estimations on rapidly increasing sequences, carry over verbatim to the current paper. To avoid inflating the contents of the paper, they have not been repeated. Section \ref{extra stuff} outlines some additional structural properties of $\X$. The final section is devoted to a detailed discussion of possible future directions in this line of research.

All vector spaces are over the complex field. This is to mirror standard practice in the study of operator algebras and it is not essential; all definitions and proofs work just as well over the real field. Denote by $\C_\Q$ the field of complex numbers with rational real and imaginary parts and put $\D_\Q = \{\la\in\C_\Q:|\la|\leq 1\}$. Let $c_{00}$ denote the vector space of eventually zero complex sequences. For $f = (b_n)_n$, $x = (a_n)_n$ in $c_{00}$ define $f(x) = \sum_nb_na_n$. The unit vector basis of $c_{00}$ is denoted both by $(e_n^*)_n$ and by $(e_n)_n$ depending on whether it is seen as a sequence of functionals or one of vectors. For $x\in c_{00}$ denote $\supp(x) = \{n\in\N:e_n^*(x)\neq 0\}$ and for $E\subset \N$ let $Ex = \sum_{n\in E}e_n^*(x)e_n$. For $x$, $y$ in $c_{00}$ write $x<y$ to mean that their supports are successive subsets of $\N$.  A sequence of successive vectors in $c_{00}$ is called a {\em block sequence}.

Let, for the entirety of this paper, $(K,\rho)$ be a fixed compact metric space.

\section{Heuristic explanation on a mixed-Tsirelson space $X_{^{C\>\!\!(\>\!\!K\>\!\!)}}$}
\label{mT section}
Schlumprecht space from \cite{schlumprecht:1991} is one of the important evolutionary steps in the history of non-classical Banach spaces. It was an integral ingredient in the solution to the unconditional sequence problem by Gowers and Maurey who constructed in \cite{gowers:maurey:1993}  the first hereditarily indecomposable (HI) space $X_{GM}$. On $X_{GM}$, every bounded linear operator is of the form $T = \la I + S$, with $S$ strictly singular. Mixed-Tsirelson spaces can be viewed as a discretization of Schlumprecht space. The first such space was defined by Argyros and Deliyanni in \cite{argyros:deliyanni:1997} and this section introduces a new variant.

With a novel approach, a Gowers-Maurey-like reflexive mixed-Tsirelson space $X_{^{C\>\!\!(\>\!\!K\>\!\!)}}$ is defined on which a large class of diagonal operators are bounded. More precisely, this space has a conditional Schauder basis and the following properties.
\begin{enumerate}[leftmargin=19pt,label=(\alph*)]

\item\label{diag+ss} Every bounded linear operator $T$ on $X_{^{C\>\!\!(\>\!\!K\>\!\!)}}$ can be written as $T = D+S$ with $D$ diagonal and $S$ strictly singular.

\item\label{C(K) quotient algebra} The quotient algebra $\mathcal{L}(X_{^{C\>\!\!(\>\!\!K\>\!\!)}})/\mathcal{SS}(X_{^{C\>\!\!(\>\!\!K\>\!\!)}})$ is isomorphic, as a Banach algebra, to $C(K)$.

\end{enumerate}
It is important to point out that not all bounded scalar sequences define bounded diagonal operators on $X_{^{C\>\!\!(\>\!\!K\>\!\!)}}$, otherwise its basis would be unconditional. Modulo the strictly singular operators, the bounded diagonal operators describe the space $C(K)$.

There already exists an approach that has been used to introduce a prescribed structure of operators in Gowers-Maurey-type spaces (see, e.g.,  \cite{gowers:1994} and \cite{gowers:maurey:1997}). In this method a space is defined by constructing a norming set, i.e., a subset $W$ of the ball $B_{X^*}$ of the dual that defines the norm of the resulting space $X$. To enforce the boundedness of a desired collection  of operators, the standard norming set $W_\mathrm{GM}$ of $X_\mathrm{GM}$ is augmented to a larger set $W$ by making it stable under the action of $T^*$ (i.e., $T^*(W)\subset W$), for $T$ in an appropriate class $\mathcal{C}$. The present approach is entirely different. It is based on a saturation under constraints from \cite{argyros:motakis:2020} and instead of enriching the set $W_\mathrm{GM}$, very strict conditions are imposed on what members are allowed to be used from it. The operators that will eventually be bounded never appear explicitly in the construction and the imposed constraints relate to weights of functions and the metric of the compact space under consideration. Rather unexpectedly, this impoverishment of the norming set results in the enrichment of the space of operators. This somewhat bizarre phenomenon is being discussed here with the sole purpose of introducing the ideas behind the Bourgain-Delbaen spaces that appear in the main result of the paper. The properties of $X_{^{C\>\!\!(\>\!\!K\>\!\!)}}$ are not justified in full detail and the only proof presented here is that certain diagonal operators are bounded.

At this point it needs to be pointed out that there are more conventional paths that yield a space satisfying properties \ref{diag+ss} and \ref{C(K) quotient algebra}. These paths however do not translate well to the Bourgain-Delbaen setting. Also, in reality the current construction does not use the set $W_\mathrm{GM}$ as a starting point but a similar standard set $W_\mathrm{mT}^\mathrm{HI}$ introduced below.

\subsection{Mixed-Tsirelson spaces}
Fix, for the remainder of this paper, a double sequence of positive even integers $(m_j,n_j)_j$ that satisfies the conditions
\begin{enumerate}[label=(\alph*)]

\item $m_1\geq 8$,

\item $m_{j+1}\geq m_j^2$,

\item $n_1\geq m_1^2$, and

\item $n_{j+1}\geq (16n_j)^{\log_2m_{j+1}}$.

\end{enumerate}

\begin{dfn}
\label{weighted def}
Let $W$ be a subset of the unit ball of $(c_{00},\|\cdot\|_\infty)$.
\begin{enumerate}[leftmargin=21pt,label=(\roman*)]

\item The set $W$ is called a {\em norming set} if it contains the unit vector basis and for every $f\in W$, $\lambda\in\D_\Q$, and interval $E$ of $\N$, $\la Ef\in W$.

\item For $j\in\N$, $W$ is said to be {\em closed under the $(m_j^{-1},\mathcal{A}_{n_j})$-operation}, if for every $1\leq d\leq n_j$ and $f_1<\cdots<f_d\in W$,
\[f = \frac{1}{m_j}\sum_{i=1}^df_i\in W.\]
Such an $f$ is said to be the outcome of an $(m_j^{-1},\mathcal{A}_{n_j})$-operation and it is called a {\em weighted functional} with $\weight(f) = m_j^{-1}$. Note that this notion is dependent on $W$ and it is not necessarily uniquely determined. As a convention it will be assumed that for all $n\in\N$ and $\lambda\in\D_\Q$, $\weight(\lambda e^*_n) = 0$. This is important.

\item\label{restricted operation} For $j\in\N$ and a family $\mathcal{F}$ of finite sequences of successive members of $W$, the set $W$ is said to be {\em closed under the $(m_j^{-1},\mathcal{A}_{n_j},\mathcal{F})$-operation} if for every $1\leq d\leq n_j$ and $(f_1,\ldots,f_d)$ in $\mathcal{F}$,
\[f = \frac{1}{m_j}\sum_{i=1}^df_i\in W.\]
Such an $f$ is said to be the outcome of an $(m_j^{-1},\mathcal{A}_{n_j},\mathcal{F})$-operation.

\end{enumerate}

\end{dfn}

\begin{dfn}
\label{basic mT}
The fundamental mixed-Tsirelson norming set $W_{\mathrm{mT}}$ is the smallest subset of $c_{00}$ that satisfies the following properties.
\begin{enumerate}[leftmargin=21pt,label=(\roman*)]

\item The set $W_{\mathrm{mT}}$ is a norming set.

\item For every positive integer $j\in\N$ the set $W_{\mathrm{mT}}$ is closed under the $(m_{j}^{-1},\mathcal{A}_{n_{j}})$-operation.

\end{enumerate}
\end{dfn}
The space $X_\mathrm{mT}$ induced by this norming set is the completion of $c_{00}$ under the norm $\|x\| = \sup\{|f(x)|:f\in W_{\mathrm{mT}}\}$. The experienced reader may have spotted that $W_{\mathrm{mT}}$ is not closed under rational convex combinations. This is intentional and it is a necessary prerequisite for the application of saturation under constraints with increasing weights in the style of \cite{argyros:motakis:2020}. This ingredient will be included in the recipe of the space $X_{^{C\>\!\!(\>\!\!K\>\!\!)}}$. The omission of convex combinations is a constant feature in this paper.

When it comes to HI-type constructions, the family of special sequences resulting from a Maurey-Rosenthal {\em coding function} (\cite{maurey:rosenthal:1977}) is a ubiquitous tool. Let $\mathcal{Q}$ denote the collection of all finite sequences of successive non-zero members of the unit ball of $(c_{00},\|\cdot\|_\infty)$ that have coefficients in $\D_\Q$. Fix an injection $\sigma:\mathcal{Q}\to \N$ so that for every $(f_1,\ldots,f_d)\in\mathcal{Q}$, $\sigma(f_1,\ldots,f_d) > \|f_d\|_\infty^{-1}2^{\max\supp(f_d)}$.

\begin{dfn}
Let $W$ be a norming set. A sequence $(f_1,\ldots,f_d)\in\mathcal{Q}$ of weighted functionals is $W$ is called a {\em special sequence} if the following hold:
\begin{enumerate}[label=(\roman*),leftmargin=19pt]

\item for some $j_1\in\N$, $\weight(f_1) = m_{4j_1-2}^{-1}$, and

\item for $1<i\leq d$ $\weight(f_i) = m_{4\sigma(f_1,\ldots,f_{i-1})}^{-1}$.

\end{enumerate}
Denote the collection of all special sequences in $W$ by $\mathcal{F}_\mathrm{sp}(W)$.
\end{dfn}
The important feature of a special sequence is that the weight of the last member uniquely determines the sequence of its predecessors.

\begin{dfn}
The fundamental mixed-Tsirelson HI norming set $W_{\mathrm{mT}}^\mathrm{hi}$ is the smallest subset of $c_{00}$ that satisfies the following properties.
\begin{enumerate}[leftmargin=21pt,label=(\roman*)]

\item The set $W_{\mathrm{mT}}^\mathrm{hi}$ is a norming set.

\item\label{even weights} For every even positive integer $2j\in\N$ the set $W_{\mathrm{mT}}^\mathrm{hi}$ is closed under the $(m_{2j}^{-1},\mathcal{A}_{n_{2j}})$-operation.

\item\label{odd weights} For every odd positive integer $2j-1\in\N$ the set $W_{\mathrm{mT}}^\mathrm{hi}$ is closed under the $(m_{2j-1}^{-1},\mathcal{A}_{n_{2j-1}},\mathcal{F}_\mathrm{sp}(W_{\mathrm{mT}}^\mathrm{hi}))$-operation.

\end{enumerate}
\end{dfn}
The space $X_\mathrm{mT}^\mathrm{hi}$ induced by this norming set is an HI space on which every bounded linear operator $T$ can be written as $T = \la I + S$ with $S$ strictly singular.

\subsection{Definition of $X_{^{C\>\!\!(\>\!\!K\>\!\!)}}$} In \cite{gowers:1994} and \cite{gowers:maurey:1997} the space of operators is augmented by loosening the above condition \ref{odd weights} and thus creating a larger norming set. Instead, in the current paper  conditions \ref{even weights} and \ref{odd weights} are tightened even further to result in a smaller norming set. As it turns out, this enriches the space of operators. The following constraint was first considered in \cite{argyros:motakis:2020}. It has its roots in papers \cite{odell:schlumprecht:1995} and \cite{odell:schlumprecht:2000} of Odell and Schlumprecht and it was further developed in a series of several papers by Argyros, the author, and others (\cite{argyros:motakis:2014}, \cite{argyros:motakis:2019}, etc.).

\begin{dfn}
Let $W$ be a norming set. A sequence of weighted functionals $f_1<f_2<\cdots<f_d$ in $W$ is said to be {\em very fast growing} if, for $1<i\leq d$, $\weight(f_i) < 2^{-\max\supp(f_{i-1})}$. Denote the collection of all very fast growing sequences in $W$ by $\mathcal{F}_\mathrm{vfg}(W)$.
\end{dfn}
Note that any sequence of basis elements is very fast growing and that $\mathcal{F}_\mathrm{sp}(W)\subset \mathcal{F}_\mathrm{vfg}(W)$. The family $\mathcal{F}_\mathrm{vfg}(W)$ yields a constraint by applying it to Definition \ref{weighted def} \ref{restricted operation}. This type of constraint has been used to study the local and asymptotic structure of Banach spaces. Such structure has strong implications to spaces of operators, e.g., in \cite{argyros:motakis:2014} the first reflexive space with the invariant subspace property was constructed.

As it was already mentioned, an additional constraint needs to be introduced which comes from the metric space $(K,\rho)$. For the remainder of this section fix a sequence $(\kappa_i)_{i=1}^\infty$ in $K$ so that for all $n\in\N$, $\{\kappa_i:i\geq n\}$ is dense in $K$. Next, for a given norming set $W$, one associates to some  $f\in W$ an element $\kappa(f)$ of $K$.

\begin{dfn}
Let $W$ be a norming set. For every $n\in\N$ and $\la\in\D_\Q$ define $\kappa(\la e^*_n) = \kappa_n$.
\begin{enumerate}[label=(\alph*),leftmargin=19pt]

\item A sequence $f_1<\cdots<f_d$ in $\mathcal{F}_\mathrm{vfg}(W)$ is said to have {\em essentially rapidly converging supports} if, for $1\leq i\leq d$, $\kappa(f_i)$ is defined and there exists $\kappa_0\in K$ such that for $1<i\leq d$, $\rho(\kappa(f_i),\kappa_0) \leq 2^{-\max\supp(f_{i-1})}$. Denote the collection of all such sequences in $W$ by $\mathcal{F}_\mathrm{ercs}(W)$.

\item For an $f$ in $W$ that is the outcome of an $(m_j^{-1},\mathcal{A}_{n_j},\mathcal{F}_\mathrm{ercs}(W))$-operation applied to a sequence $(f_i)_{i=1}^d$ as above, define $\kappa(f) = \kappa_0$.

\end{enumerate}
\end{dfn}
Note that a sequence with essentially rapidly converging supports is always assumed to be very fast growing. This has notational advantages but it would also have been fine to disentangle the two notions. Similar to the weight function, the associated element $\kappa(f)$ is not unique and $\kappa$ is a partially defined multi valued function from $W$ to $K$.  Also note that the definition of $\kappa$ is implicit. This is formally sound and $\kappa$ is, at the very least, defined on the basis elements. So if a norming set $W$ is closed under infinitely many $(m_j^{-1},\mathcal{A}_{n_j},\mathcal{F}_\mathrm{ercs}(W))$-operations then the compactness of $K$ yields a wealth of functionals $f$ for which $\kappa(f)$ is defined. Indeed, any very fast growing sequence $(f_n)_n$, for which all $\kappa(f_n)$ are defined, has a subsequence with essentially rapidly converging supports on which the operations can be applied.

For a norming set $W$ denote $\mathcal{F}_\mathrm{ercs}^\mathrm{sp}(W) = \mathcal{F}_\mathrm{ercs}(W)\cap\mathcal{F}_\mathrm{sp}(W)$, i.e., the collection of special sequences with essentially rapidly converging supports. The time is ripe to define the norming set $W_{^{C\>\!\!(\>\!\!K\>\!\!)}}$ of the space $X_{^{C\>\!\!(\>\!\!K\>\!\!)}}$.

\begin{dfn}
Define $W_{^{C\>\!\!(\>\!\!K\>\!\!)}}$ to be the smallest subset of $c_{00}$ that satisfies the following properties.
\begin{enumerate}[label=(\roman*),leftmargin=21pt]

\item The set $W_{^{C\>\!\!(\>\!\!K\>\!\!)}}$ is a norming set.

\item For every even positive integer $2j\in\N$ the set $W_{^{C\>\!\!(\>\!\!K\>\!\!)}}$ is closed under the $(m_{2j}^{-1},\mathcal{A}_{n_{2j}}, \mathcal{F}_\mathrm{ercs}(W_{^{C\>\!\!(\>\!\!K\>\!\!)}}))$-operation.

\item For every odd positive integer $2j-1\in\N$ the set $W_{^{C\>\!\!(\>\!\!K\>\!\!)}}$ is closed under the $(m_{2j-1}^{-1},\mathcal{A}_{n_{2j-1}}, \mathcal{F}^\mathrm{sp}_\mathrm{ercs}(W_{^{C\>\!\!(\>\!\!K\>\!\!)}}))$-operation.

\end{enumerate}
\end{dfn}

It is useful to observe that the set $W_{^{C\>\!\!(\>\!\!K\>\!\!)}}$ can be defined as an increasing union of sets $W_m$, $m=0,1,\ldots$ where $W_0 = \{\la e_n: n\in\N\text{ and }\la\in\D_\Q\}$ and if $W_m$ has been defined then $W_{m+1}$ is the union of $W_m$ with the collection of all $\la Ef$, where $\la\in\D_\Q$, $E$ is an interval of $\N$, and $f$ is the outcome of an $(m_{2j}^{-1},\mathcal{A}_{n_{2j}}, \mathcal{F}_\mathrm{ercs}(W_m))$-operation or an $(m_{2j}^{-1},\mathcal{A}_{n_{2j}}, \mathcal{F}^\mathrm{sp}_\mathrm{ercs}(W_m))$-operation. This in particular implies that for every $f\in W_{^{C\>\!\!(\>\!\!K\>\!\!)}}$, $\kappa(f)$ is defined and that for every $\la\in\D_\Q$ and interval $E$ of $\N$ such that $\la Ef\neq 0$, $\kappa(\la Ef) = \kappa(f)$.

It is almost shocking that this norming set induces a space $X_{^{C\>\!\!(\>\!\!K\>\!\!)}}$ that is not HI. Unless $K$ is a singleton, $X_{^{C\>\!\!(\>\!\!K\>\!\!)}}$ contains decomposable subspaces (see Proposition \ref{unconditionality in X_K} \ref{hi singleton}). For a continuous function $\phi:K\to\C$ denote by $\hat \phi:c_{00}\to c_{00}$ the linear operator with $\hat\phi(e_n) = \phi(\kappa_n)e_n$. Note that $\hat \phi$, being a diagonal operator, is formally dual to itself, i.e., for every $f,x\in c_{00}$, $f(\hat\phi x) = (\hat\phi f)(x)$.

\begin{prop}
\label{mT diagonal bounded}
Let $\phi:K\to\C$ be a Lipschitz function. Then $\hat\phi$ extends to a bounded linear operator on $X_{^{C\>\!\!(\>\!\!K\>\!\!)}}$.
\end{prop}

With regards, to the proof of Proposition \ref{mT diagonal bounded}, the fact that $\phi$ is Lipschitz is not particularly important and by tweaking the metric (before defining $X_{^{C\>\!\!(\>\!\!K\>\!\!)}}$) any continuous function may be assumed Lipschitz. The trueness of this result stems from the fact that every $f$ in $W_{^{C\>\!\!(\>\!\!K\>\!\!)}}$ has a perturbation $\tilde f$ so that the set $\{\kappa_n:n\in\supp(\tilde f)\}$ has small diameter in $K$. Therefore $\phi$ is almost constant on this subset and thus $\hat\phi f$ is close to a scalar multiple of $f$. Crucially, there is no unique scalar that works for all $f$. The following statement makes this more precise while simultaneously yielding Proposition \ref{mT diagonal bounded}.

\begin{prop}
\label{mT diagonal bounded precise}
Let $\phi:K\to\C$ be a Lipschitz function. Then, there exists $N\in\N$ so that for every $f\in W_{^{C\>\!\!(\>\!\!K\>\!\!)}}$ with $\min\supp(f)\geq N$,
\begin{equation}
\label{mT diagonal bounded precise eq0}
\Big\|\hat\phi f - \phi\big(\kappa(f)\big)f\Big\| \leq 3\weight(f)\|\phi\|_\infty.
\end{equation}
Therefore, on $X_{^{C\>\!\!(\>\!\!K\>\!\!)}}$, $\|\hat\phi - P_{[1,N)} \hat\phi \| \leq 2\|\phi\|_\infty$. 
\end{prop}

With this result at hand, and the Stone-Weierstrass theorem, it is not hard to see that $C(K)$ embeds isomorphically, as a Banach algebra, into $\mathcal{L}(X_{^{C\>\!\!(\>\!\!K\>\!\!)}})/\mathcal{SS}(X_{^{C\>\!\!(\>\!\!K\>\!\!)}})$ (here, it is necessary to use that for all $n\in\N$, $\{\kappa_i:i\geq n\}$ is dense in $K$). Although Proposition \ref{mT diagonal bounded precise} will be proved now, it will not be shown here that the aforementioned embedding is onto. The interested reader may be able to extrapolate this information from the Bourgain-Delbaen part of this paper.

\begin{proof}[Proof of Proposition \ref{mT diagonal bounded precise}]
Assume, without loss of generality, that $\|\phi\|_\infty = 1$ and denote by $L$ the Lipschitz constant of $\phi$. Pick $N\in\N$ sufficiently large so that $(3+L)/2^{N-1}\leq 5/8$. Statement \eqref{mT diagonal bounded precise eq0} is proved by induction on $m=0,1,\ldots$ for all $f\in W_m$ with $\min\supp(f)\geq N$. For $m = 0$ and $f = \la e_n\in W_0$, $\hat\phi f = \phi(\kappa_n)e_n = \phi(\kappa(f))f$ and thus the conclusion holds.

Assume next that \eqref{mT diagonal bounded precise eq0} is true for all $g\in W_m$ with $\min\supp(g) \geq N$ and let $f\in W_{m+1}$ with $\min\supp(f)\geq N$. Then, there exist $j\in\N$, $d\leq n_j$, and a sequence $(f_i)_{i=1}^d$ in $W_m$ with essentially rapidly converging supports in $[N,\infty)$ such that $f = m_j^{-1}\sum_{i=1}^df_i$. Then,
\begin{align*}
\Big\|\hat\phi f - \phi(\kappa(f))f\Big\| \span= \frac{1}{m_j}\Big\|\sum_{i=1}^d\big(\hat\phi(f_i) - \phi(\kappa(f))\big)f_i\Big\|\\
\leq& \frac{1}{m_j}\Big(\sum_{i=1}^d\Big\|\big(\hat\phi(f_i) - \phi(\kappa(f_i))\big)f_i\Big\| + \sum_{i=1}^d\Big|\phi(\kappa(f)) - \phi(\kappa(f_i))\Big|\|f_i\|\Big)\\
\leq& \frac{1}{m_j}\Big(3\weight(f_1) +  \sum_{i=2}^d3\weight(f_i) +2 + \sum_{i=2}^d L \rho(\kappa(f),\kappa(f_i))\Big)\\
\leq& \frac{1}{m_j}\Big(\frac{19}{8} + 3\sum_{i=2}^d 2^{-\max\supp(f_{i-1})} + L\sum_{i=2}^d2^{\max\supp(f_{i-1})}\Big)\\
\leq&\frac{1}{m_j}\big(19/8 + 5/8\big) = 3\weight(f).
\end{align*}
\end{proof}

The Bourgain-Delbaen construction that is about to follow is based on the same principles. In order to achieve an isometric result, some components of the definition are chosen more with more precision.

\section{The Bourgain-Delbaen $\mathscr{L}_\infty$-space $\X$}
\label{BD section}

A separable Banach space $X$ is a $\mathscr{L}_{\infty,C}$-space, where $C\geq 1$, if there exists an increasing sequence $(F_n)_n$ of finite dimensional subspaces of $X$, the union of which is dense in $X$ and so that for all $n\in\N$, $F_n$ is $C$-isomorphic to $\ell_\infty^{\mathrm{dim}(F_n)}$. Suppressing the constant $C$, $X$ is called a $\mathscr{L}_\infty$-space. The class of $\mathscr{L}_p$-spaces was introduced by Lindenstrauss and Pe\l czy\'nski in \cite{lindenstrauss:pelczynski:1968}. Bourgain and Delbaen introduced in \cite{bourgain:delbaen:1980} a method for constructing non-classical separable $\mathscr{L}_\infty$-spaces. It is one of the essential components in the solution of the scalar-plus-compact problem by Argyros and Haydon in \cite{argyros:haydon:2011} (the other being a mixed-Tsirelson implementation of the hereditarily indecomposable Gowers-Maurey space). The purpose of the first part of this section is to recall a very general Bourgain-Delbaen scheme that is based on \cite{argyros:gasparis:motakis:2016}, where it was proved that every separable $\mathscr{L}_\infty$-space is isomorphic to a Bourgain-Delbaen space. Following this introduction (and following in the footsteps of \cite{argyros:haydon:2011}), a Bourgain-Delbaen space modeled after the space $X_\mathrm{mT}$ is introduced. Finally, the extra ingredients discussed in Section \ref{mT section} are adjusted to this setting to define the space $\X$.

\subsection{General Bourgain-Delbaen $\mathscr{L}_\infty$-spaces} At the most basic level the idea behind the Bourgain-Delbaen construction method is very elegant. For two non-empty sets $\Ga_1\subset\Ga$ denote by $r_{_{\Ga_1}}:\ell_\infty(\Ga)\to\ell_\infty(\Ga_1)$ the usual restriction operator. Any linear right inverse $i:\ell_\infty(\Ga_1)\to\ell_\infty(\Ga)$ of $r_{_{\Ga_1}}$ will be called an {\em extension operator}, i.e., $i$ is a linear operator so that for all $x\in\ell_\infty(\Ga_1)$ and $\gamma\in\Ga_1$, $i(x)(\gamma) = x(\gamma)$.

The Bourgain-Delbaen scheme is an infinite inductive process in which one defines finite sets $\Ga_1\subset\Ga_2\subset\Ga_3\subset\cdots$ and extension operators $i_{1,2}:\ell_\infty(\Ga_1)\to \ell_\infty(\Ga_2)$, $i_{2,3}: \ell_\infty(\Ga_2)\to\ell_\infty(\Ga_3)$,\ldots. In the base step, one picks $\Ga_1$ and no extension operator. Having defined $\Ga_1,\ldots,\Ga_n$ and $i_{1,2},\ldots,i_{n-1,n}$, denote $\De_1 = \Ga_1,\De_2 = \Ga_2\setminus\Ga_1,\ldots,\De_n = \Ga_n\setminus\Ga_{n-1}$. Note that $\De_1,\ldots,\De_n$ are pairwise disjoint finite sets and $\Ga_i = \cup_{j=1}^i\De_j$, for $1\leq i\leq n$. To perform the $(n+1)$'th step, one chooses a finite set $\De_{n+1}$, that is disjoint from all previously defined ones, and an extension operator $i_{n,n+1}:\ell_\infty(\Ga_{n})\to\ell_\infty(\Ga_{n+1})$, where $\Ga_{n+1} = \Ga_n\cup\De_{n+1}$. Although this is the basic essence of the scheme, the inductive choice needs to be performed in a very special manner (described later) to achieve something of interest.

Note that for every $m<n$ this method yields an extension operator $i_{m,n} = i_{n-1,n}\circ\cdots\circ i_{m+1,m+2}\circ i_{m,m+1}:\ell_\infty(\Ga_m)\to\ell_\infty(\Ga_n)$. Also, denote $i_{n,n} = id:\ell_\infty(\Ga_n)\to\ell_\infty(\Ga_n)$. The condition under which this construction yields a $\mathscr{L}_\infty$-space is the following.

\begin{ass}
\label{BD assumption}
There exists $C\geq 1$ so that for every $m\leq n$, $\|i_{m,n}\| \leq C$.
\end{ass}

If this has been achieved, putting $\Ga = \cup_{q=1}^\infty\Ga_q$, for each $m\in\N$ define the extension operator $i_m = \lim_n i_{m,n} : \ell_\infty(\Ga_m)\to\ell_\infty(\Ga)$. Then, for all $n\in\N$, $\|i_n\| \leq C$ and, because $i_n$ is a right inverse of $r_{_{\Ga_n}}$, the space $Y_n = i_n(\ell_\infty(\Ga_n))$ is $C$-isomorphic to $\ell_\infty(\Ga_n)$. It also follows that $Y_1\subset Y_2\subset\cdots$ and therefore the space $\mathfrak{X}_{(\Ga_n,i_n)} = \overline{\cup_{n=1}^\infty Y_n}\subset\ell_\infty(\Ga)$ is a $\mathscr{L}_{\infty,C}$-space. Any space resulting from such a process is called a Bourgain-Delbaen $\mathscr{L}_\infty$-space.

On such a space $\mathfrak{X}_{(\Ga_n,i_n)}$, the extension operators are used to define a finite dimensional decomposition (FDD). For each $n\in\N$ the map $P_n = i_nr_{_{\Ga_n}}: \mathfrak{X}_{(\Ga_n,i_n)}\to Y_n$ is a projection of norm at most $C$. In fact $P_nP_m = P_{n\wedge m}$ and thus the sequence of spaces $Z_1 = P_1(\mathfrak{X}_{(\Ga_n,i_n)})=i_1(\ell_\infty(\De_1))$, $Z_n = (P_n-P_{n-1})(\mathfrak{X}_{(\Ga_n,i_n)}) = i_n(\ell_{\infty}(\Delta_n))$, $n\geq 2$, forms a FDD of $\mathfrak{X}_{(\Ga_n,i_n)}$. Denote, for all $m<n\in\N$, $P_{(m,n]} = P_n - P_m$ the associated projection onto the space $Z_{m+1}\oplus\cdots\oplus Z_n$ by. For any interval $I$ of $\N$ define $P_I$ analogously. It is also true that if, for all $n\in\N$ and $\ga\in\De_n$, one defines $d_\ga = i_n(e_\ga)$, then the sequence $((d_\ga)_{\ga\in\De_n})_{n=1}^\infty$ forms a Schauder basis of $\mathfrak{X}_{(\Ga_n,i_n)}$ (see \cite[Remark 2.10, page 688]{argyros:gasparis:motakis:2016}).

While carrying out the Bourgain-Delbaen construction, of particular importance are specific versions of the above projections that can be defined during the steps of the induction. For every $m\leq n\in\N$, and once the $n$'th step is complete, put $P_m^{(n)} = i_{m,n}r_{_{\Ga_m}}:\ell_\infty(\Ga_n)\to Y_m^{(n)} = i_{m,n}(\ell_\infty(\Ga_m))$. For an interval $I$ of $\{1,\ldots,n\}$ define $P^{(n)}_{I}$ analogously.

\subsection{Bourgain-Delbaen extension functionals}
In the $(n+1)$'th inductive step and having chosen $\Ga_1,\ldots,\Ga_n$ and $i_{1,2},\ldots,i_{n-1,n}$ one must define the set $\De_{n+1}$ and an extension operator $i_{n,n+1}:\ell_\infty(\Ga_{n})\to\ell_\infty(\Ga_{n+1})$. Presupposing that the index set $\De_{n+1}$ has been determined, defining $i_{n,n+1}$ is equivalent to finding linear functionals $c_\gamma^*:\ell_\infty(\Ga_{n})\to\C$, $\gamma\in\De_{n+1}$, so that for all $x\in\ell_\infty(\Ga_{n})$ and $\gamma\in\De_{n+1}$, $i_{n,n+1}(x)(\gamma) = c_\gamma^*(x)$. Thus, one may shift their focus on defining  $(c_\gamma^*)_{\gamma\in\De_{n+1}}$ instead of $i_{n,n+1}$ directly. For obvious reasons, each such $c_\ga^*$ is called an {\em extension functional}. 

Although, formally, for each $\ga\in\Ga_{n+1}$, $c_\ga^*$ is defined on $\ell_\infty(\Ga_n)$ in the end it can also be viewed as functional on $\ell_\infty(\Ga)$, and thus on $\mathfrak{X}_{(\Ga_n,i_n)}$ if Assumption \ref{BD assumption} is satisfied. This is done by identifying $c_\ga^*$ with $c_\ga^*\circ r_{_{\Ga_n}}$. Make the convention that for every $\ga\in\Ga_1$, $c_\ga^* = 0$ (this is natural as these are not truly extension functionals). By setting, for each $\ga\in\Ga$, $d_\ga^* = e_\ga^* - c_\ga^*$ it turns out that $d_\ga^*(d_{\ga'}) = \delta_{\ga,\ga'}$, i.e., $(d_\ga,d_\ga^*)_{\ga\in\Ga}$ forms a biorthogonal system in $\mathfrak{X}_{(\Ga_n,i_n)}\times\mathfrak{X}_{(\Ga_n,i_n)}^*$ (see \cite[Proposition 2.17 (i), page 690]{argyros:gasparis:motakis:2016}).

Bourgain and Delbaen pointed out in \cite[Lemma 4.1, page 161]{bourgain:delbaen:1980} that if the extension functionals $c_\ga^*$ are of a certain form, then Assumption \ref{BD assumption} will be automatically satisfied. Here, a specific case of this form is borrowed from the Argyros-Haydon construction in \cite{argyros:haydon:2011}. The following notation will be used frequently. For $n\in\N$ and $\gamma\in\De_n$, write $\rank(\ga) = n$.

\begin{prop}
\label{BD formula}
Assume that for every $n\in\N$ and $\gamma\in\Delta_{n+1}$ the functional $c_\gamma^*$ is either zero or of one of the following forms.
\begin{enumerate}[leftmargin=19pt,label=(\alph*)]

\item\label{age zero} There exist an interval $I\subset\{1,\ldots,n\}$,  $\eta\in\Ga_n$, $\la\in\D$, and  $j\in\N$ such that
\[c_\ga^* =  \frac{1}{m_j} \la e^*_\eta\circ P^{(n)}_{I}.\]

\item\label{age nonzero} There exist $\xi\in\Ga_{n-1}$, and interval $I$ of $\{\rank(\xi)+1,\ldots,n\}$, $\eta\in\Ga_n$ with $\rank(\eta) >\rank(\xi)$, $\la\in\D$,  and $j\in\N$ such that
\[c_\ga^* = e^*_\xi + \frac{1}{m_j}\la e^*_\eta\circ P^{(n)}_{I}.\]

\end{enumerate}
Then, for every $m\leq n\in\N$, $\|i_{m,n}\| \leq m_1/(m_1-2)$ and thus Assumption \ref{BD assumption} is satisfied.
\end{prop}

\begin{proof}
See \cite[Lemma 4.1, page 161]{bourgain:delbaen:1980} or \cite[Theorem 3.5, page 11]{argyros:haydon:2011}.
\end{proof}

\begin{rem}
\label{component norms}
If the assumptions of Proposition \ref{BD formula} are satisfied then
\begin{enumerate}[leftmargin=21pt,label=(\roman*)]

\item for all $m\leq n\in\N$, $\|P_{[1,n]}\| \leq 2$ and  $\|P_{[m,n]}\|\ \leq 3$ and

\item for all $\ga\in\Ga$, $\|d_\ga\| \leq 2$ and $\|d_\ga^*\|\leq 3$.

\end{enumerate}
Indeed, $P_{[1,n]} = i_{n}r_{_{\Ga_n}}$ and $\|i_n\|\ \leq m_1/(m_1-2) \leq 4/3$. Also, $d_\ga = i_{\rank(\ga)}(e_\ga)$ and $d_\ga^* = e_\ga^*\circ P_{[\rank(\ga),\infty)}$. 
\end{rem}

\subsection{A Mixed-Tsirelson Bourgain-Delbaen $\mathscr{L}_\infty$-space $\mathfrak{X}_\mathrm{mT}$}
As it was explained in \cite{argyros:haydon:2011}, a Bourgain-Delbaen $\mathscr{L}_\infty$-space can be constructed by specifying a set of instructions (i.e., an algorithm) that takes as input disjoint index sets $\De_1,\ldots,\De_n$ (with perhaps additional information encoded in them via certain functions) and returns an index set $\Delta_{n+1}$ together with functionals $c_\ga^*$, $\ga\in\De_{n+1}$, that adhere to the assumptions of Proposition \ref{BD formula}.

The following set of instructions defines a Bourgain-Delbaen space $\mathfrak{X}_\mathrm{mT} = \mathfrak{X}_{(\bar\Ga_n,\bar i_n)}$ that is based on $X_\mathrm{mT}$ from Section \ref{mT section}. The bar-notation is used to differentiate the objects associated to $\mathfrak{X}_\mathrm{mT}$ from the ones associated to the final space $\X$. To facilitate the forthcoming construction of the space $\X$, there will be an additional involved parameter, namely an (as of yet unspecified) increasing sequence $(K_n)_n$ of finite subsets of $K$. For each $n\in\N$, denote $\D_n = \{(k_1/2^n)e^{i (k_2/2^n)2\pi}: k_1,k_2 = 1,\ldots,2^n\}$. Each constructed $\gamma$ will have additional information encoded in it, namely a weight and an age.

\begin{inst}
Put
\[\bar\Ga_1 = \bar\De_1 =  \{(1,\kappa):\kappa\in K_1\}\]
and for $\gamma \in \bar\De_1$ define $\weight(\gamma) = 0$ and leave $\age(\gamma)$ undefined.

Assume that $\bar\Ga_1,\ldots,\bar\Ga_n$ as well as $\bar i_{1,2},\ldots,\bar i_{n-1,n}$ have been defined. Also assume that a function $\weight:\bar\Ga_n\to\{m_j^{-1}:j\in\N\}\cup\{0\}$ and a partially defined function $\age:\bar\Ga_n\to\N$ have been constructed.

The set $\bar\De_{n+1}$ is defined as the disjoint union of sets $\bar\De^0_{n+1}$, $\bar\De^\text{\ref{age zero}}_{n+1}$, and $\bar\De^\text{\ref{age nonzero}}_{n+1}$. First, put
\begin{equation*}
\bar\De_{n+1}^0 = \Big\{(n+1,\kappa): \kappa\in K_{n+1}\Big\}
\end{equation*}
and for $\ga\in\bar\De_{n+1}$ define $\weight(\ga) = 0$, leave $\age(\ga)$ undefined, and put $\bar c_\ga^* = 0$.

Let
\begin{align*}
\bar\De_{n+1}^\text{\ref{age zero}} = \Big\{&(n+1,I,\eta,\la,m_j^{-1},\kappa):\;I\text{ is an inteval of }\{1,\ldots,n\},\\
&\eta\in\bar\Ga_n,\;\la \in\D_{n+1},\;j\in\{1,\ldots,n+1\},\text{ and  }\kappa\in K_{n+1}\Big\}.
\end{align*}
For  $\ga = (n+1,I,\eta,\la,m_j^{-1},\kappa)\in\bar\De^\text{\ref{age zero}}_{n+1}$ define $\weight(\ga) = m_j^{-1}$, $\age(\ga) = 1$, and
\[\bar c_\ga^* = \frac{1}{m_j}\la \bar e^*_\eta\circ \bar P_I^{(n)}.\]
Note that for $\ga\in\bar\De_{n+1}^\text{\ref{age zero}}$, $\bar c_\ga^*$ is of type \ref{age zero} from Proposition \ref{BD formula}.

Next, let
\begin{align*}
\bar\De_{n+1}^\text{\ref{age nonzero}} = \Big\{&(n+1,\xi,I,\eta,\la,m_j^{-1},\kappa):\;\xi\in \bar\Ga_{n-1}\text{ with }\weight(\xi) = m_j^{-1}\text{ and}\\
&\age(\xi)<n_j,\;I\text{ is an inteval of }\{\rank(\xi)+1,\ldots,n\},\\
&\eta\in\bar\Ga_n,\;\la \in\bar\D_{n+1},\;j\in\{1,\ldots,n+1\},\text{ and  }\kappa\in K_{n+1}\Big\}.
\end{align*}
For $(n+1,\xi,I,\eta,\la,m_j^{-1},\kappa)\in\bar\De_{n+1}^\text{\ref{age nonzero}}$ define $\weight(\ga) = m_j^{-1}$, $\age(\ga) = \age(\xi)+ 1$, and
\[\bar c_\ga^* =\bar e_\xi^* + \frac{1}{m_j}\la \bar e_\eta^* \bar P_I^{(n)}.\]
Note that for $\ga\in\bar\De_{n+1}^\text{\ref{age nonzero}}$, $\bar c_\ga^*$ is of type \ref{age nonzero} from Proposition \ref{BD formula}.

Define $\bar\De_{n+1} = \bar\De_{n+1}^0\cup \bar\De_{n+1}^\text{\ref{age zero}}\cup \bar\De_{n+1}^\text{\ref{age nonzero}}$.
\end{inst}

This set of instructions defines the Bourgain-Delbaen $\mathscr{L}_{\infty}$-space $\mathfrak{X}_\mathrm{mT}$. Note that the space $\mathfrak{X}_\mathrm{mT}$ is similar to the space $\mathfrak{B}_\mathrm{mT}$ from \cite{argyros:haydon:2011}. There are three differences. The first one is that here, for each $n\in\N$, there are several copies of each extension functional (even the zero one) indexed over $K_n$. The second difference is that here no convex combinations of $e_\eta^*$ are allowed. As already mentioned, this is necessary to perform a saturation under constraints with increasing weights in the flavour of \cite{argyros:motakis:2020}. The final difference is that in \cite{argyros:haydon:2011} there was no need to use the entire collection of intervals $I$ used here. This full collection was also used in the construction in \cite{argyros:motakis:2019}.

The claimed connection to the space $X_\mathrm{mT}$ is made a lot clearer by the {\em evaluation analysis} of each $\ga\in\bar\Ga$. This concept is from \cite{argyros:haydon:2011}. Note that every $\ga\in\Ga$ with non-zero $\weight(\ga) = m_j^{-1}$ has an $\age(\ga) = a\leq n_j$. 

\begin{prop}
\label{BDmT evaluation analysis}
Let $\ga\in\bar\Ga$ have $\weight(\ga) = m_j^{-1}$ and $\age(\ga) = a\leq n_j$. Then, there exist
\begin{enumerate}[label=(\roman*),leftmargin=22pt]

\item $\xi_1,\ldots,\xi_a\in\bar\Ga$ with $j\leq\rank(\xi_1)<\cdots<\rank(\xi_a)$ and $\xi_a = \ga$,

\item intervals $I_1$ of $[1,\rank(\xi_1))$ and $I_r$ of $(\rank(\xi_{r-1}),\rank(\xi_r))$, $r=2,\ldots,a$,

\item $\eta_r\in\bar\Ga$ with $\rank(\eta_r) < \rank(\xi_r)$, for $r=1,\ldots,a$, and

\item $\la_r\in\D_{\rank(\xi_r)}$, for $r=1,\ldots,a$,

\end{enumerate}
such that
\[\bar e_\ga^* = \sum_{r=1}^a\bar d_{\xi_r}^* + \frac{1}{m_j}\sum_{r=1}^a \la_r \bar e_{\eta_r}^*\circ \bar P_{I_r}.\]
\end{prop}

\begin{proof}
This is proved by induction on $\age(\ga)$, by writing $\bar e_\ga^* = \bar d_\ga^* + \bar c_\ga^*$ and unravelling the definition of $\bar c_\ga^*$. See \cite[Proposition 4.5, page 15]{argyros:haydon:2011} for further details.
\end{proof}

Intuitively, each $\bar e_\ga^*$ can be obtained as the outcome of a certain kind of $(m_j^{-1},\mathcal{A}_{n_j})$-operation applied to other $\bar e_\eta^*$.

\subsection{Definition of $\X$}
To define the space $\X$, an appropriate subset $\Ga$ of $\bar \Ga$ will be chosen. Recall that bar-notation is used for objects relevant to $\mathfrak{X}_\mathrm{mT}$ whereas objects without bars are related to $\X$. Start with the sets $(\bar\De_n)_n$. Put $\De_1 = \bar \De_1$ and for each $n\geq 2$ choose a suitable $\De_n\subset\bar\De_n$. Among other specialized properties, the set $\Ga = \cup_n\De_n$ will be a self-determined subset of $\bar\Ga = \cup_n\bar\De_n$. This means that for every $n\in\N$ and $\ga\in\De_{n+1}$, whichever $\xi$ or $\eta$ appear in the defining formula of $\bar c_\ga^*$ must be in $\Ga_n$. For example, if $\ga\in\De_{n+1}\cap \bar\De_{n+1}^\text{\ref{age nonzero}}$ and $\bar c_\ga^* = \bar e_\xi^* + m_j^{-1}\la \bar e^*_\eta\circ\bar P^{(n)}_I$, then $\xi,\eta\in \Ga_n$. Then the functional $c_\ga^* = e_\xi^* + m_j^{-1}\la e^*_\eta\circ P^{(n)}_I:\ell_\infty(\Ga_n)\to\C$ is well defined and is of type \ref{age nonzero} from Proposition \ref{BD formula}. A similar situation occurs for $\ga\in\De_{n+1}$ that are in $\bar\De_{n+1}^\text{\ref{age zero}}$ or $\bar\De^0_{n+1}$. Therefore, this process yields a Bourgain-Delbaen $\mathscr{L}_\infty$-space $\mathfrak{X}_{(\Ga_n,i_n)}$, which will be the space $\X$. This technique was also used by Tarbard in \cite{tarbard:2012} and \cite{tarbard:2013}, by Argyros and the author  in \cite{argyros:motakis:2019}, and by Manoussakis, Pelczar-Barwacz, and  \'Swi\c etek in \cite{manoussakis:pelczar:swietek:2017}. It was formulated as a method for general  Bourgain-Delbaen spaces by Argyros and the author in \cite{argyros:motakis:2019}. According to \cite[Proposition 1.12, page 1893]{argyros:motakis:2019}, the restriction map $r_{_\Ga}$ on $\mathfrak{X}_\mathrm{mT}$ is a quotient map onto $\X$. This process is analogous to the fact that $W_{^{C\>\!\!(\>\!\!K\>\!\!)}}$ is built as a subset of $W_\mathrm{mT}$.

The choice of $\Ga$ requires the enforcement of certain constraints onto its member $\gamma$. The first one comes from a Maurey-Rosenthal coding function applied to the construction of $c^*_\ga$ when $\gamma$ has odd weight. Therefore, fix an injection $\sigma:\bar\Ga\to \N$ such that for every $\gamma\in\bar\Ga$, $m_{4\sigma(\ga)} >  2^{\rank(\ga)}$.

The second constraint pertains to the weight of $\eta$, which needs to be sufficiently small, whenever it appears in the definition of a given $\ga$. For this, it is necessary to define an additional weight function. For every interval $I$ of $\N$ and $\ga\in\bar\Ga$ define
\begin{equation}
\label{new weight}
\weight(I,\ga) = \left\{
	\begin{array}{ll}
	0 & \mbox{if } \rank(\ga) \leq \min(I),\\
	\weight(\ga) & \mbox{otherwise.}
	\end{array}
\right.
\end{equation}
The utility of this is that in the end, for every $\ga\in\Ga$, if $\rank(\ga) = \min(I)$ then $e_\ga^*\circ P_I = d_\ga^*$. This, being a basis element, models the behaviour of basis elements in $W_{^{C\>\!\!(\>\!\!K\>\!\!)}}$, which by convention have zero weight. If $\rank(\ga)<\min(I)$, then $e_\ga^*\circ P_I = 0$.

The final constraint relates to the metric space $K$ and is encoded in the sets $(K_n)_n$ which were used as a parameter in the construction of $\bar \Ga$. Recall that for every $\ga\in\bar\Ga$, with $\rank(\ga) = n$, its last coordinate is some $\kappa\in K_n$. Define the function $\kappa:\bar\Ga\to\cup_n K_n\subset K$ that retrieves from each $\ga$ its final coordinate $\kappa = \kappa(\ga)$. This will be used as follows. For any $\ga$ and for whichever $\xi$, $\eta$ appear in the defining formula of $c_\ga^*$, the elements $\kappa(\ga)$, $\kappa(\xi)$, and $\kappa(\eta)$ must be close to one another in $K$.

To specify the sequence of sets $(K_n)_n$, fix a countable dense subset $\mathcal{A}$ of $C(K)$ that is closed under $\C_\Q$-linear combinations. Enumerate the set $\mathcal{B} = \{\phi\in\mathcal{A}:\|\phi\|_\infty\leq 1\}$ as $\{\phi_1,\phi_2,\ldots\}$. For all $n\in\N$ define
\begin{equation*}
%\label{Bn}
\mathcal{B}_n = \Big\{\theta^k\phi_{i_1}\phi_{i_2}\cdots \phi_{i_k}:k\in\N, 0\leq \theta\leq\frac{n}{n+1},\text{ and }1\leq i_1,i_2,\ldots,i_k\leq n\Big\}.
\end{equation*}
Then, $\mathcal{B}_1\subset\mathcal{B}_2\subset\cdots$ and each $\mathcal{B}_n$ is a compact multiplicative subsemigroup of the unit ball of $C(K)$. Define the equivalent metric on $K$
\begin{equation*}
%\label{better metric}
\varrho(\kappa,\kappa') = \sum_{n=1}^\infty2^{-n}\max\Big\{\Big|\phi(\kappa) - \phi(\kappa')\Big|:\phi\in\mathcal{B}_n\Big\}.
\end{equation*}
This change of metric will make it possible to renorm the space $\X$ to achieve an isometric result. What will eventually be required is that $(\mathcal{B}_n)_n$ is increasing and that for each $n\in\N$ every $\phi\in\mathcal{B}_n$ is $2^n$-Lipschitz and $\mathcal{B}_n$ is a multiplicative semigroup. Note that the equivalence of the metric follows from the compactness of the sets $\mathcal{B}_n$, which will not be used again. Now, fix an increasing sequence of finite subsets $(K_n)_n$ of $K$ so that, for each $n\in\N$, $K_n$ is an $m_n^{-1}2^{-(n+1)}$-net of $K$ (with respect to $\varrho$).

It is time to define the space $\X$.

\begin{inst}
\label{main def}
Define $\Ga_1 = \De_1 = \bar \De_1$. Assume that $\Ga_1,\ldots,\Ga_n$ as well as $i_{1,2},\ldots, i_{n-1,n}$ have been defined.

The set $\De_{n+1}$ is defined as the disjoint union of sets $\De^0_{n+1}$, $\De^\text{\ref{age zero},even}_{n+1}$, $\De^\text{\ref{age zero},odd}_{n+1}$, $\De^\text{\ref{age nonzero},even}_{n+1}$, and $\De^\text{\ref{age nonzero},odd}_{n+1}$. First, put $\De^0_{n+1} = \bar\De^0_{n+1}$. Let
\begin{align*}
\De_{n+1}^\text{\ref{age zero},even} = \Big\{&(n+1,I,\eta,\la,m_j^{-1},\kappa)\in\bar\De_{n+1}^\text{\ref{age zero}}:\;j\in(2\N)\text{ and }\eta\in\Ga_n\Big\}\\
\De_{n+1}^\text{\ref{age zero},odd} = \Big\{&(n+1,I,\eta,\la,m_j^{-1},\kappa)\in\bar\De_{n+1}^\text{\ref{age zero}}:\;j\in(2\N-1),\\
&\eta\in\Ga_n\text{ and }\weight(\eta)=m^{-1}_{4i-2}\text{, for some }i\in\N\Big\}.
\end{align*}
Put $\De^\text{\ref{age zero}}_{n+1} = \De^\text{\ref{age zero},even}_{n+1}\cup\De^\text{\ref{age zero},odd}_{n+1}$ and for  $\ga = (n+1,I,\eta,\la,m_j^{-1},\kappa)\in\De^\text{\ref{age zero}}_{n+1}$ define
\[ c_\ga^* = \frac{1}{m_j}\la e^*_\eta\circ P_I^{(n)}.\]
Note that for $\ga\in\De^\text{\ref{age zero}}_{n+1}$, $c_\ga^*$ is of type \ref{age zero} from Proposition \ref{BD formula}.

Next, let
\begin{align*}
\De_{n+1}^\text{\ref{age nonzero},even} = \Big\{&(n+1,\xi,I,\eta,\la,m_j^{-1},\kappa)\in\bar\De_{n+1}^\text{\ref{age nonzero}}:\;j\in(2\N),\;\xi\in\Ga_{n-1},\;\eta\in\Ga_n\\
&\text{with }\weight(I,\eta) \leq 2^{-\rank(\xi)},\;\varrho(\kappa(\xi),\kappa)\leq m_j^{-1}2^{-\rank(\xi)},\\
&\text{and }\varrho(\kappa(\eta),\kappa) \leq 2^{-\rank(\xi)}\Big\}.\\
\De_{n+1}^\text{\ref{age nonzero},odd} = \Big\{&(n+1,\xi,I,\eta,\la,m_j^{-1},\kappa)\in\bar\De_{n+1}^\text{\ref{age nonzero}}:\;j\in(2\N-1),\;\xi\in\Ga_{n-1},\\
&\eta\in\Ga_n\text{ with }\weight(\eta) = m^{-1}_{4\sigma(\xi)},\;\varrho(\kappa(\xi),\kappa)\leq m_j^{-1}2^{-\rank(\xi)},\\
&\text{and }\varrho(\kappa(\eta),\kappa) \leq 2^{-\rank(\xi)}\Big\}.
\end{align*}
Put $\De_{n+1}^\text{\ref{age nonzero}} =\De_{n+1}^\text{\ref{age nonzero},even} \cup \De_{n+1}^\text{\ref{age nonzero},odd}$ and  for $(n+1,\xi,I,\eta,\la,m_j^{-1},\kappa)\in\De_{n+1}^\text{\ref{age nonzero}}$ define
\[c_\ga^* =  e_\xi^* + \frac{1}{m_j}\la e_\eta^* P_I^{(n)}.\]
Note that for $\ga\in\De_{n+1}^\text{\ref{age nonzero}}$, $c_\ga^*$ is of type \ref{age nonzero} from Proposition \ref{BD formula}.

Define $\De_{n+1} = \De_{n+1}^0\cup \De_{n+1}^\text{\ref{age zero}}\cup \De_{n+1}^\text{\ref{age nonzero}}$ and denote $\X = \mathfrak{X}_{(\Ga_n,i_n)}$.
\end{inst}

This time, the evaluation analysis of a $\ga\in\Ga$ can be supplemented with information about weights and metric distance. This information resembles the conditions imposed on $W_{^{C\>\!\!(\>\!\!K\>\!\!)}}$.

\begin{prop}
\label{X_K evaluation analysis}
Let $\ga\in\Ga$ have $\weight(\ga) = m_j^{-1}$ and $\age(\ga) = a\leq n_j$. Then, there exist
\begin{enumerate}[label=(\roman*),leftmargin=23pt]

\item $\xi_1,\ldots,\xi_a\in\Ga$ with $j\leq\rank(\xi_1)<\cdots<\rank(\xi_a)$ and $\xi_a = \ga$,

\item intervals $I_1$ of $[1,\rank(\xi_1))$ and $I_r$ of $(\rank(\xi_{r-1}),\rank(\xi_r))$, $r=2,\ldots,a$,

\item $\eta_r\in\Ga$ with $\rank(\eta_r) < \rank(\xi_r)$, for $1=2,\ldots,a$, and

\item $\la_r\in\D_{\rank(\xi_r)}$, $r=1,\ldots,a$,

\end{enumerate}
such that
\begin{equation}
\label{X_K evaluation analysis1}
e_\ga^* = \sum_{r=1}^a d_{\xi_r}^* + \frac{1}{m_j}\sum_{r=1}^a \la_r e_{\eta_r}^*\circ P_{I_r}
\end{equation}
and for $r=2,\ldots,a$,
\begin{align}
\label{X_K evaluation analysis1}\weight(I_r,\eta_r) &\leq 2^{-\rank(\xi_{r-1})},\\
\label{X_K evaluation analysis2}\varrho\big(\kappa(\xi_{r-1}),\kappa(\xi_r)\big) &\leq m_j^{-1}2^{-\rank(\xi_{r-1})},\text{ and}\\
\label{X_K evaluation analysis3}\varrho\big(\kappa(\eta_r),\kappa(\xi_r)\big) &\leq 2^{-\rank(\xi_{r-1})}
\end{align}
The representation \eqref{X_K evaluation analysis1} is called the {\em evaluation analysis of $\gamma$}.
\end{prop}

\begin{proof}
This is proved by induction on $\age(\ga) = a$. If $\age(\ga) = 1$, put $\xi_1 = \ga$ and write $c_\ga^* = m_j^{-1}\la_1 e_{\eta_1}^*\circ P_{I_1}$. Therefore one has  $e_\ga^* = d_\ga^* + c_\ga^* = d_{\xi_1}^* + m_j^{-1}\la_1e_{\eta_1}^*\circ P_{I_1}$. If the statement holds for all $\ga\in\Ga$ with $\age(\ga) = a$, let $\ga\in\Ga$ with $\age(\ga) = a+1$. Put $\xi_{a+1} = \ga$ and write $c_\ga^* = e^*_{\xi_{a}} + m_j^{-1}\la_{a+1}e_{\eta_{a+1}}^*$. By the definition of $\De_{n+1}^\text{\ref{age nonzero}}$,
\begin{align*}
\weight(I_{a+1},\ga_{a+1}) &\leq 2^{-\rank(\xi_a)},\\
\varrho(\kappa(\xi_{a}),\kappa(\xi_{a+1})) &\leq m_j^{-1}2^{\rank(\xi_a)},\text{ and}\\
\varrho\big(\kappa(\eta_{a+1}),\kappa(\xi_{a+1})\big) &\leq 2^{-\rank(\xi_a)}.
\end{align*}
Write $e_\ga^* = d_{\xi_{a+1}}^* + c_{\xi_{a+1}}^* = d_{\xi_{a+1}}^* + m_j^{-1}\la_{a+1}e_{\eta_{a+1}}^* + e^*_{\xi_{a}}$ and recover the remaining information from the inductive hypothesis applied to $e^*_{\xi_a}$.
\end{proof}

Similar to $W_{^{C\>\!\!(\>\!\!K\>\!\!)}}$, the set $\Ga$ is in a certain sense closed under the $(m_{2j}^{-1},\mathcal{A}_{n_{2j}})$-operations applied to sequences $(\la_re^*_{\eta_r}\circ P_{I_r})_r$ that satisfy a condition similar to having essentially rapidly converging supports.

\begin{prop}
\label{gamma builder}
Let $j\in\N$, $(q_r)_{r=1}^{n_{2j}}$ be a strictly sequence of natural numbers, $(I_r)_{r=1}^{n_{2j}}$ be finite intervals of $\N$, $\eta_r\in\Ga$ and $\la_r\in\D_{q_r}$, for $r=1,\ldots,n_{2j}$, and $\kappa_0\in K$. Assume that the following are satisfied.
\begin{enumerate}[leftmargin=23pt,label=(\roman*)]

\item $2j\leq q_1$, $I_1\subset[1,q_1)$, and $I_r\subset (q_{r-1},q_r)$ , for $r=2,\ldots,n_{2j}$.

\item $\rank(\eta_r)<q_r$, for $r=1,\ldots,n_{2j}$,

\item $\weight(I_r,\eta_r) \leq 2^{-q_{r-1}}$, for $r=2,\ldots,n_{2j}$.

\item $\varrho(\kappa(\eta_k),\kappa_0) \leq 2^{-(q_{r-1}+1)}$, for $r=2,\ldots,n_{2j}$.

\end{enumerate}
Then, there exists $\ga\in\Ga$ with $\weight(\ga) = m_{2j}^{-1}$, $\rank(\ga) = q_{n_{2j}}$, and $\varrho(\kappa(\ga),\kappa_0) \leq 2^{-q_{n_{2j}}}$ such that $e_\ga^*$ has an evaluation analysis
\[e_\ga^* = \sum_{r=1}^{n_{2j}}d_{\xi_r}^* + \frac{1}{m_{2j}}\sum_{r=1}^{n_{2j}}\la_r e_{\eta_r}^*\circ P_{I_r},\]
where $\rank(\xi_r) = q_{r}$, for $r=1,\ldots,n_{2j}$.
\end{prop}

\begin{proof}
Recall that $K_{q_r}$ is a $m_{q_r}^{-1}2^{-(q_r+1)}$-net of $K$ and thus it is possible to pick $\kappa_r\in K_{q_r}$ with $\varrho(\kappa_r,\kappa_0)\leq m_{q_r}^{-1}2^{-(q_r + 1)}$. In particular, for $r=2,\ldots,n_{2j}$, $\varrho(\kappa_r,\kappa_{r-1})\leq m_{2j}^{-1}2^{-q_{r-1}}$. Next, choose inductively $\xi_1,\ldots,\xi_{n_{2j}} = \ga$ that satisfy the following properties.
\begin{enumerate}[label=(\alph*),leftmargin=19pt]

\item $\rank(\xi_r) = q_r$, for $r=1,\ldots,n_{2j-1}$

\item $\kappa(\xi_r) = \kappa_r$, for $r=1,\ldots,n_{2j-1}$

\item $c_{\xi_1}^* = m_{2j}^{-1}\la_1 e_{\ga_1}^*\circ P_{I_1}$ and $c_{\xi_r}^* = e_{\xi_{r-1}}^* + m_{2j}^{-1}\la_r e_{\eta_r}\circ P_{I_r}^*$, for $r=2,\ldots,n_{2j-1}$.

\end{enumerate}
Note that $\xi_1 = (q_1,I_1,\eta_1,\la_1,m_{2j}^{-1},\kappa_1)$ meets all the required conditions to be in $\De_{q_1}^{\text{\ref{age zero},even}}$. If, for some $2\leq n_{2j-1}$, $\xi_{r-1}$ has been chosen then, $\varrho(\kappa(\eta_r),\kappa_r) \leq \varrho(\kappa(\eta_r),\kappa_0) + \varrho(\kappa_r,\kappa_0)\leq  2^{-q_{r-1}}$. Therefore, $\xi_r = (q_r,\xi_{r-1},I_r,\eta_r,\la_r,m_{2j}^{-1},\kappa_r)$ is in $\De_{q_r}^{\text{\ref{age nonzero},even}}$. It is the straightforward to check that $\ga = \xi_{n_{2j}}$ has the desired weight and proximity to $\kappa_0$. The fact that it also has the desired evaluation analysis follows from the proof of Proposition \ref{X_K evaluation analysis}.
\end{proof}

\begin{rem}
A similar process can be carried out to construct $\ga$ of odd weight $m_{2j}^{-1}$. The difference is that in each step the resulting $\xi_r$ determines, via the coding function $\sigma$, the weight $m^{-1}_{4\sigma(\xi_r)}$ that ${\eta_{r+1}}$ is allowed to have.
\end{rem}

\section{Diagonal operators on $\X$ and their image in $\mathpzc{Cal}(\X)$}
\label{diagonal boundedness section}
For every continuous function $\phi:K\to\C$ denote by $\hat \phi$ the linear operator on $\langle\{d_\ga:\ga\in\Ga\}\rangle$ given by $\hat\phi(d_\ga) = \phi(\kappa(\ga))d_\ga$. In this section it will be proved that whenever $\phi$ is Lipschitz then $\hat \phi$ extends to a bounded linear operator on $\X$. The derived estimates yield a natural Banach algebra embedding $\Psi: C(K)\to\mathpzc{Cal}(\X)$. At the end of this section it will be shown that $\X$ admits an equivalent norm that turns this into a homomorphic isometric embedding. All the results of this section deal with the ``unconditional'' structure of $\X$, i.e., the special properties of odd-weight functionals are not used. These will be necessary further down the road when it will be established that $\Psi$ is onto as well.

For $\phi$ as above also define $\hat\phi$ on $\langle\{d_\ga^*:\ga\in\Ga\}\rangle\subset \X^*$ using the same formula, i.e., for all $\ga\in\Ga$, $\hat\phi(d^*_\ga) = \phi(\kappa(\ga))d^*_\ga$. Under this notation, for all $x\in\langle\{d_\ga:\ga\in\Ga\}\rangle$ and $f\in \langle\{d_\ga^*:\ga\in\Ga\}\rangle$, $f(\hat \phi x) = (\hat\phi f)(x)$. That is, $\hat\phi$ is its own dual. Note that for every $\ga\in\Ga$ and interval $I$ of $\N$, $e_\ga^*\circ P_I$ is in the linear span of $\{d_{\ga'}^*:\ga'\in\Ga\}$ (see, e.g., \cite[Proposition 2.17 (ii), page 690]{argyros:gasparis:motakis:2016}). Therefore, $\hat \phi(e_\ga^*\circ P_I)$ is always well defined.

\subsection{Boundedness of diagonal operators} The following statement is analogous to Proposition \ref{mT diagonal bounded precise}. The additional precision will be required later in this section to define a renorming of $\X$.

\begin{prop}
\label{lip diag bounded}
For every $L,M\geq 0$ there exists $N\in\N$ with the following property. For every $\phi:K\to\C$ that is $L$-Lipschitz with $\|\phi\|_\infty\leq M$, for every $\ga\in\Ga$, and for every interval $I$ of $\N$ with $\min(I) \geq N$,
\begin{equation}
\label{lip diag bounded1}
\Big\|\hat \phi\big(e^*_\ga\circ P_I\big) - \phi\big(\kappa(\ga)\big)e_\ga^* \circ P_I\Big\| \leq 7\weight(I,\ga)M.
\end{equation}
In particular, $\hat\phi:\X\to\X$ is bounded and $\big\|\hat \phi -  P_{[1,N)} \hat \phi\big\| \leq 4M$.
\end{prop}

The proof comes down to taking the evaluation analysis of a given $\ga$ and applying an inductive hypothesis to its components, just like in the proof of Proposition \ref{mT diagonal bounded precise}. For the sake of tidiness, some computations have been isolated and gathered in the following.

\begin{lem}
\label{telescoping}
Let $\ga\in\Ga$ have $\weight(\ga) = m_j^{-1}$ and $\age(\ga) = a>1$. Following the notation of Proposition \ref{X_K evaluation analysis}, for every $1\leq r_0<a$,
\begin{align}
\label{telescoping1}\sum_{r=r_0+1}^a\weight(I_r,\eta_r) &\leq 2\cdot 2^{-\rank(\xi_{r_0})},\\
\label{telescoping2}\sum_{r=r_0}^{a}\varrho\big(\kappa(\xi_{r}),\kappa(\ga)\big) &\leq \frac{4}{m_j}2^{-\rank(\xi_{r_0})},\text{ and}\\
\label{telescoping3}\sum_{r=r_0+1}^a\varrho\big(\kappa(\eta_{r}),\kappa(\ga)\big) &\leq 3\cdot2^{-\rank(\xi_{r_0})}.
\end{align}
\end{lem}

\begin{proof}
For \eqref{telescoping1}, use \eqref{X_K evaluation analysis1}.
\begin{equation*}
\sum_{r=r_0+1}^a\weight(I_r,\eta_r) \leq \sum_{r = r_0+1}^a2^{-\rank(\xi_{r-1})} \leq \sum_{s=0}^\infty2^{-(\rank(\xi_{r_0})+s)} = 2\cdot2^{-\rank(\xi_{r_0})}.
\end{equation*}
Next, to prove \eqref{telescoping2}, \eqref{X_K evaluation analysis2} is used.
\begin{align*}
\sum_{r=r_0}^{a}\varrho\big(\kappa(\xi_{r}),\kappa(\ga)\big) &=\sum_{r=r_0}^{a-1}\varrho\big(\kappa(\xi_{r}),\kappa(\ga)\big) \leq \sum_{r=r_0}^{a-1}\sum_{s=r+1}^a\varrho\big(\kappa(\xi_{s-1}),\kappa(\xi_s)\big)\\
& \stackrel{\eqref{X_K evaluation analysis2}}{\leq}\frac{1}{m_j}\sum_{r=r_0}^{a-1}\sum_{s=r+1}^a2^{-\rank(\xi_{s-1})} {\leq}\frac{1}{m_j}\sum_{r=r_0}^{a-1}\sum_{s=0}^\infty 2^{-(\rank(\xi_r)+s)}\\
&\leq \frac{2}{m_j}\sum_{r=r_0}^{a-1}2^{-\rank(\xi_r)} \leq  \frac{4}{m_j} \cdot2^{-\rank(\xi_{r_0})}.
\end{align*}
Finally, to obtain \eqref{telescoping3}, use \eqref{X_K evaluation analysis3}.
\begin{align*}
\sum_{r=r_0+1}^a\varrho\big(\kappa(\eta_{r}),\kappa(\ga)\big) & \leq \sum_{r=r_0+1}^a\varrho\big(\kappa(\eta_{r}),\kappa(\xi_r)\big) + \sum_{r=r_0+1}^a\varrho\big(\kappa(\xi_{r}),\kappa(\ga)\big)\\
&\!\!\!\!\!\!\stackrel{\eqref{X_K evaluation analysis3}\text{\&}\eqref{telescoping2}}{\leq}\sum_{r=r_0+1}^a2^{-\rank(\xi_{r-1})} + \frac{4}{m_j}\cdot2^{-\rank(\xi_{r_0+1})}\\
&\leq 2\cdot2^{-\rank(\xi_{r_0})} + \frac{4}{6}\cdot2^{-\rank(\xi_{r_0})}.
\end{align*}
\end{proof}

\begin{proof}[Proof of Proposition \ref{lip diag bounded}]
The second part follows from the first one and the fact that for any $x\in\X$ with $\|x\|\leq 1$,
\begin{align*}
\|P_{[N,\infty)}\hat\phi(x)\| &= \sup_{\ga\in\Ga}|(\hat\phi(e_\ga^*\circ P_{[N,\infty)}))(x)|\\
&\leq \sup_{\ga\in\Ga}\lim_m\Big(\big|\phi\big(\kappa(\ga)\big)\big|\Big\|e_\ga^*\Big(P_{[N,m)}x\Big)\Big\| + 7\weight([N,m),\ga)M\Big)\\
&\leq \sup_{\ga\in\Ga}\Big(M\Big\|e_\ga^*\circ P_{[N,\infty)}\Big\| + (7/8)M\Big) \leq 4M.
\end{align*}
Choose $N\in\N$ such that $(21(L/M)+14)/2^N\leq 1/8$ (if $M=0$ then there is nothing to prove). Observe that whenever $\weight(I,\ga) = 0$ then either $e_\ga^*\circ P_I = d_\ga^*$ or $e_\ga^*\circ P_I = 0$. In either case, $\hat\phi(e_\ga^*\circ P_I) = \phi(\kappa(\ga))e_\ga^*\circ P_I$ and in particular \eqref{lip diag bounded1} holds. For the remaining cases, \eqref{lip diag bounded1} is proved by induction on $\rank(\ga)$. The case $\rank(\ga) = 1$ is covered by the fact that all such $\ga$ have zero weight.

Assume now that \eqref{lip diag bounded1} holds for all $\ga$ with $\rank(\ga)\leq n$. Let $\ga\in\Ga$ with $\rank(\ga) = n+1$ and let $I$ be an interval of $\N$.  If $\weight(I,\ga) = 0$ the desired conclusion holds. Assume therefore that $\weight(I,\ga) = \weight(\ga) =  m_j^{-1}$. Apply Proposition \ref{X_K evaluation analysis} and write

\begin{align*}
e_\ga^*\circ P_I &= \sum_{r=1}^ad_{\xi_r}^*\circ P_I + \frac{1}{m_j}\sum_{r=1}^a\la_re_{\eta_r}^*\circ P_{I_r}\circ P_I.\\
& = \sum_{\rank(\xi_r)\in I}d_{\xi_r}^* + \frac{1}{m_j}\sum_{r=r_0}^{a_0}\la_r e_{\eta_r}^*\circ P_{I_r\cap I},
\end{align*}
where $r_0 = \min\{r:\rank(\xi_r)\geq \min(I)\}$ (which is well defined because $\weight(I,\ga)>0$) and $a_0 = \min\{r:\rank(\xi_r) = a\text{ or }\rank(\xi_r)\geq\max(I)\}$. Therefore,

\begin{align}
\label{want to bound}\span\Big\|\hat \phi\big(e^*_\ga\circ P_I\big) - \phi\big(\kappa(\ga)\big)e_\ga^* \circ P_I\Big\| = \Big\|\sum_{\rank(\xi_r)\in I}\Big(\phi\big(\kappa(\xi_r)\big)-\phi\big(\kappa(\ga\big)\Big)d_{\xi_r}^*\\
 &\phantom{=}+ \frac{1}{m_j}\sum_{r=r_0}^{b_0}\la_r \Big(\hat\phi\big( e_{\eta_r}^*\circ P_{I_r\cap I}\big) - \phi\big(\kappa(\ga)\big)e_{\eta_r}^*\circ P_{I_r\cap I}\Big)\Big\|\nonumber\\
 &\leq L\sum_{\rank(\xi_r)\in I}\varrho\big(\kappa(\xi_r),\kappa(\ga)\big)\|d_{\xi_r}^*\|\label{mess1}\\
 &\phantom{=} + \frac{1}{m_j}\sum_{r=r_0}^{a_0}\Big\|\hat\phi\big( e_{\eta_r}^*\circ P_{I_r\cap I}\big) - \phi\big(\kappa(\eta_r)\big)e_{\eta_r}^*\circ P_{I_r\cap I}\Big\|\label{mess2}\\
 &\phantom{=}+ \frac{1}{m_j}\sum_{r=r_0}^{a_0}\Big|\phi\big(\kappa(\eta_r)\big)-\phi\big(\kappa(\ga)\big)\Big|\big\|e_{\eta_r}^*\circ P_{I_r\cap I}\big\|\label{mess3}.
 \end{align}
 It is better to treat the three above terms separately. By \eqref{telescoping2},
 \[\eqref{mess1}\leq 3L\frac{4}{m_j}2^{-\rank(\xi_{r_0})} \leq \frac{1}{m_j}\frac{12(L/M)}{2^N}M.\]
For the next term, use the inductive hypothesis to obtain

\begin{align*}
\eqref{mess2}&\leq \frac{1}{m_j}\sum_{r=r_0}^{b_0}7\weight(I_r\cap I,\eta_r)M\\
& \stackrel{\eqref{telescoping1}}{\leq} \frac{1}{m_j}7\Big(\weight(I_{r_0}\cap I,\eta_{r_0}) + 2\cdot 2^{\rank(\xi_{r_0})}\Big)M\\
&\leq \frac{1}{m_j}\Big(\frac{7}{8} + \frac{14}{2^N}\Big)M.
\end{align*} 
Evaluate the third term.

\begin{align*}
\eqref{mess3}&\leq \frac{1}{m_j}3\Big(2M + L\sum_{r=r_0+1}^{a_0}\varrho\big(\kappa(\eta_r),\kappa(\gamma)\big)\Big)\\
&\stackrel{\eqref{telescoping3}}{\leq}\frac{1}{m_j}\Big(6 + \frac{9(L/M)}{2^{N}}\Big)M
\end{align*}
Combining the estimates for \eqref{mess1}, \eqref{mess2}, and \eqref{mess3} one obtains
\[\eqref{want to bound}\leq \frac{1}{m_j}\Big(\frac{7}{8} + 6+ \frac{21(L/M)+14}{2^N}\Big)M \leq 7\weight(\ga)M.\]
\end{proof}

\subsection{The embedding $\Psi:C(K)\to\mathpzc{Cal}(\X)$}
Denote by $\mathrm{Lip}(K)$ the algebra of all Lipschitz function $\phi:K\to\C$. By the Stone-Weierstrass theorem, $\mathrm{Lip}(K)$ is dense in $C(K)$. Proposition \ref{lip diag bounded} yields that the map $\widehat{\,\cdot\,}:\mathrm{Lip}(K)\to \mathcal{L}(\X)$ is a well defined (but unbounded) linear homomorphism. Denote by $[\,\cdot\,]:\mathcal{L}(\X)\to\mathcal{L}(\X)/\mathcal{K}(\X) = \mathpzc{Cal}(\X)$ the quotient map. The map $\Psi:\mathrm{Lip}(K)\to \mathpzc{Cal}(\X)$ given by $\Psi(\phi) = [\hat\phi]$ is a bounded homomorphism and it is shown here that it is also an embedding. It is then shown that, by renorming $\X$, $\Psi$ can be turned into a homomorphic isometric embedding.

\begin{prop}
\label{embedding formula}
The map $\Psi:\mathrm{Lip}(K)\to \mathpzc{Cal}(\X)$ extends to a homomorphic embedding $\Psi :C(K)\to \mathpzc{Cal}(\X)$.

In fact, for every equivalent norm $\trn{\cdot}$ on $\X$, the map $\Psi:C(K)\to \mathpzc{Cal}(\X,\trn{\cdot})$ is noncontractive.
\end{prop}

\begin{proof}
By Proposition \ref{lip diag bounded}, the map $\Psi$ is well defined and $\|\Psi\| \leq 4$. To verify the last statement, fix an equivalent norm $\trn{\cdot}$ on $\X$ and $\phi\in\mathrm{Lip}(K)$. Take $\kappa \in\cup_n K_n$ and let $Y_\kappa = \overline{\langle\{d_\ga:\kappa(\ga) = \kappa\}\rangle}$. This space is infinite dimensional; for $n$ sufficiently large $\gamma = (n,\kappa)\in\De^0_{n}$ and $\kappa(\ga) = \kappa$. By definition, $\hat \phi(x) = \phi(\kappa)x$, for all $x\in Y_\kappa$. Therefore, for all $A\in\mathcal{K}(\X)$, $\trnsmall{(\hat\phi - A)|_{Y_\kappa}} \geq |\phi(\kappa)|$. Because $\cup_nK_n$ is dense in $K$, $\trnsmall{\hat\phi - A}\geq\|\phi\|_\infty$.
\end{proof}

To turn $\Psi$ into an isometric embedding the semigroups $(\mathcal{B}_n)_n$ are used.

\begin{prop}
\label{renorming}
There exists an equivalent norm $\trn{\cdot}$ on $\X$ such that the map $\Psi:C(K)\to\mathpzc{Cal}(\X,\trn{\cdot})$ is a homomorphic isometry.
\end{prop}

\begin{proof}
By Proposition \ref{embedding formula}, it suffices to find a norm $\trn{\cdot}$ on $\X$ that makes $\Psi$ nonexpansive. Recall that, for $n\in\N$, each $\phi\in\mathcal{B}_n$ has norm at most one and is $2^n$-Lipschitz. Let $N(n)$ denote the minimum $N$ given by Proposition \ref{lip diag bounded}, for $L=2^n$ and $M=1$. This means that for each $m\leq n\in\N$ and $\phi\in\mathcal{B}_m$, $\|P_{[N(n),\infty)}\hat\phi \|\leq 4$. For each $x\in\N$ define $\|x\|_0 = \|x\|$,
\[\|x\|_n = \sup\Big\{\big\|P_{[N(n),\infty)}\hat\phi(x)\big\|:\phi\in\mathcal{B}_n\Big\}, \text{ for }n\in\N,\]
and $\trn{x} = \sup_{n\geq 0}\|x\|_n$. Then, $\|x\| \leq \trn{x}\leq 4\|x\|$. Recall that $\mathcal{A}$ is a dense $\C_\Q$-linear subspace of $C(K)$ and that $\mathcal{B} = \{\phi\in\mathcal{A}:\|\phi\|_\infty\leq 1\} = \{\phi_1,\phi_2,\ldots\}$. To complete the proof, it suffices to fix $\phi\in \mathcal{B}$ and show $\inf_{A\in\mathcal{K}(\X)}\trnsmall{\hat \phi - A} \leq 1$.

Let $n_0\in\N$ such that $\phi = \phi_{n_0}$. For $n\geq n_0$ it will be now shown that $\trnsmall{\hat\phi - P_{[1,N(n))}\hat\phi }\leq \frac{n+1}{n}$. To that end, let $x\in\X$ with $\trn{x}\leq 1$. First, to compute $\|\hat\phi P_{[N(n),\infty)}x\|_0$ note that $\frac{n}{n+1}\phi\in\mathcal{B}_n$. Therefore,
\[\|P_{[N(n),\infty)}\hat\phi x\|_0 = \frac{n+1}{n}\Big\|P_{[N(n),\infty)}\widehat{\big(\textstyle{\frac{n}{n+1}}\phi\big)} x\Big\| \leq \frac{n+1}{n}\|x\|_n \leq \frac{n+1}{n}\trn{x}.\]
Next, take $m\in\{1,\ldots,n\}$ and $\psi\in\mathcal{B}_m$. Then, $\hat \phi$, $\hat\psi$, $P_{[N(m),\infty)}$, $P_{[N(m),\infty)}$ all commute and $\frac{n}{n+1}\phi\psi\in\mathcal{B}_n$. This yields
\[\Big\|P_{[N(m),\infty)}\hat\psi \Big(P_{[N(n),\infty)}\hat\phi x\Big)\Big\| = \frac{n+1}{n}\Big\|P_{[N(n),\infty)}\widehat{\big(\textstyle{\frac{n}{n+1}}\phi\psi\big)}x\Big\|\leq\frac{n+1}{n}\|x\|_n.\]
Therefore, $\|P_{[N(n),\infty)}\hat\phi x\|_m\leq \frac{n+1}{n}\trn{x}$. Finally, take $m>n$ and $\psi\in\mathcal{B}_m$. Then, $\frac{n}{n+1}\phi\psi\in\mathcal{B}_m$. This yields,
\[\Big\|P_{[N(m),\infty)}\hat\psi P_{[N(n),\infty)}\hat\phi x\Big\| = \frac{n+1}{n}\Big\|P_{[N(m),\infty)}\widehat{\big(\textstyle{\frac{n}{n+1}}\phi\psi\big)}x\Big\|\leq\frac{n+1}{n}\|x\|_m,\]
i.e., $\|P_{[N(n),\infty)}\hat\phi x\|_m \leq  \frac{n+1}{n}\trn{x}$.
\end{proof}

\begin{rem}
Unless $K$ is finite, it is not true that for every continuous function $\phi:K\to\C$ the linear map $\hat \phi:\langle\{d_\ga:\ga\in\Ga\}\rangle \to\langle\{d_\ga:\ga\in\Ga\}\rangle$ extends to a bounded linear operator on $\X$. If this were the case then, by the closed graph theorem, $\widehat{\cdot}: C(K)\to\mathcal{L}(\X)$ would be a homomorphic embedding. By Proposition \ref{L(X) prop} \ref{c0 in L(X)} this is impossible.
\end{rem}

\section{The impact of the conditional structure of $\X$}
\label{impact section}
This relatively brief section discusses the outcomes of the conditional structure of $\X$ (the properties derived from the definition of odd-weight members $\ga$ of  $\Ga$). These outcomes are presented in the form of two black-box Theorems that can be used to directly prove that the embedding $\Psi:C(K)\to\mathpzc{Cal}(\X)$ is onto. The proofs of these Theorems are based on HI techniques and are included in Sections  \ref{common concepts section} and \ref{operators section}.

\begin{thm}
\label{eventual continuity}
Let $T:\X\to\X$ be a bounded linear operator. Then, for every $\e>0$ there exist $n\in\N$ and $\delta>0$ such that for all $\ga$,$\ga'\in\Ga$ with $\min\{\rank(\ga),\rank(\ga')\}\geq n$ and $\varrho(\kappa(\ga),\kappa(\ga'))<\delta$,
\begin{equation}
\label{eventual continuity eq}\big|d_\ga^*\big(Td_\ga\big) - d_{\ga'}^*\big(Td_{\ga'}\big)\big| <\e.
\end{equation}
Therefore, the function $\varphi_T:K\to \C$ given by
\[\varphi_T(\kappa) = \lim_{\substack{\rank(\ga)\to\infty\\\kappa(\ga)\to\kappa}}d_\ga^*\big(Td_\ga\big)\]
is well defined and continuous.
\end{thm}

Property \eqref{eventual continuity eq} can be seen as an eventual continuity of the diagonal entries $(d_\ga^*(Td_\ga))_{\ga\in\Ga}$ of $T$. Its proof goes deeply into the details of HI techniques. The fact that $\varphi_T$ is a well defined continuous function is a straightforward consequence of the density of $\cup_nK_n$ and some elementary real analysis. It then immediately follows that the linear map $\widetilde \Phi:\mathcal{L}(\X)\to C(K)$, given by $\widetilde \Phi(T) = \varphi_T$ is a bounded linear operator. Indeed, for any $T\in\mathcal{L}(\X)$ and $\gamma\in\Ga$, $|d_\ga^*(Td_\ga)| \leq \|d_\ga^*\|\|d_\ga\|\|T\|\leq 6\|T\|$. The next statement yields the remaining necessary information to complete the proof of Theorem \ref{main theorem}.

\begin{thm}
\label{compact zero diagonal}
A bounded linear operator $T:\X\to\X$ is compact if and only if $\varphi_T = 0$.
\end{thm}

An immediate consequence is that the map $\Phi:\mathpzc{Cal}(\X) \to C(K)$, defined via the formula $\Phi([T]) = \varphi_T$, is a bounded linear injection. Indeed, $\Phi = \widetilde\Phi/\mathrm{ker}\widetilde\Phi$. Now, an almost straightforward computation (together with Proposition \ref{renorming}) yields the main result of this paper.

\begin{cor}[Theorem \ref{main theorem} \ref{main identification}  \& \ref{main isometry}]
\label{cor main thm}
The map $\Phi:\mathpzc{Cal}(\X)\to C(K)$ is the inverse of $\Psi:C(K)\to\mathpzc{Cal}(X)$. Therefore, with an appropriate equivalent norm $\trn{\cdot}$, the Calkin algebra of $(\X,\trn{\cdot})$ is homomorphically isometric to $C(K)$.
\end{cor}

\begin{proof}
Since both $\Phi$ and $\Psi$ are injections, proving that $\Phi\circ\Psi:C(K)\to C(K)$ is the identity map yields the conclusion. By density,  it is sufficient to show that for $\phi\in \mathrm{Lip}(K)$, $\Phi(\Psi\phi) = \phi$. Recall that $\hat\phi:\X\to\X$ is a well defined bounded linear operator and for all $\ga\in\Ga$,
\[\varphi_{\widehat\phi}(\kappa) = \lim_{\substack{\rank(\ga)\to\infty\\\kappa(\ga)\to\kappa}}d_\ga^*\big(\hat\phi(d_\ga)\big) = \lim_{\substack{\rank(\ga)\to\infty\\\kappa(\ga)\to\kappa}}\phi\big(\kappa(\ga)\big) = \phi(\kappa).\]
In other words, $\Phi(\Psi\phi) = \Phi([\hat\phi]) = \widetilde\Phi(\hat\phi) = \varphi_{\widehat\phi} = \phi$.
\end{proof}

An interesting corollary is that the space $\X$ has the diagonal-plus-compact property.

\begin{cor}[Theorem \ref{main theorem} \ref{main diag+comp}]
\label{diag+comp}
Every bounded linear operator $T:\X\to\X$ is of the form $D+A$, where $D$ is diagonal bounded linear operator and $A$ is a compact linear operator.
\end{cor}

\begin{proof}
Let $T\in\mathcal{L}(\X)$ and pick an absolutely summable sequence (in the $\|\cdot\|_\infty$-norm) of Lipschitz functions $\phi_n:K\to\C$ with $\sum_{n=1}^\infty\phi_n = \varphi_T$. Using Proposition \ref{lip diag bounded}, for each $n\in\N$, pick $N(n)$ such that $\|P_{[N(n),\infty)}\hat\phi_n\|\leq 4\|\phi_n\|_\infty$. Then, $D = \sum_{n=1}^\infty P_{[N(n),\infty)}\hat\phi_n$ is a well defined diagonal bounded linear operator and
\begin{align*}
\varphi_D = \widetilde\Phi(D) &= \sum_{n=1}^\infty\widetilde\Phi\big(P_{[N(n),\infty)}\hat\phi_n\big) = \sum_{n=1}^\infty\widetilde\Phi\big(\hat\phi_n\big) = \sum_{n=1}^\infty\phi_n = \varphi_T.
\end{align*}
This means $\varphi_{T-D} = \varphi_T - \varphi_D = 0$, i.e., $A = T-D$ is compact.
\end{proof}

\begin{rem}
The expert reader will find the following fact interesting. Given a $T\in\mathcal{L}(\X)$, Theorem \ref{eventual continuity} provides an explicit diagonal operator $D$ so that $T-D$ is compact. Therefore, Kakutani's fixed point theorem (\cite{kakutani:1941}) is not required to prove non-constructively that such a $D$ must exist. This is a crucial fact in this paper as it allows the present method to work despite omitting convex combinations in the definition of $\Ga$. This fixed point theorem was used by Gowers and Maurey in \cite{gowers:maurey:1997} and by Tarbard in \cite{tarbard:2013}.
\end{rem}

\section{Common concepts from HI methods}
\label{common concepts section}
This section  goes through the ubiquitous notions of rapidly increasing sequences, exact pairs, and dependent sequences. These are specialized vectors and sequences of vectors that have been used in almost all HI and related constructions. In particular, they were also used in \cite{argyros:haydon:2011}. Estimates of the norms of linear combinations of such objects are fundamental in the study of the geometry of these spaces and of their bounded linear operators. These estimates are dependent on notions such as an auxiliary space and a basic inequality, originating in \cite{argyros:deliyanni:1997} but found in \cite{argyros:haydon:2011} as well. The versions of the proofs found in \cite{argyros:haydon:2011} work in the current setting as well. Rather than repeating pages of identical arguments, it was chosen to refer to that paper while only highlighting the minor differences.

Henceforth, for a vector $x\in\X$ the support of $x$ is the set $\supp(x) = \{n\in\N:P_{\{n\}}x\neq0\}$. Similarly, the range of $x$ is the smallest interval $\range(x)$ of $\N$ containing $\supp(x)$. A bock sequence $(x_k)$ in $\X$ is one for which $\max\supp(x_k) < \min\supp(x_{k+1})$, for all $k$.

\subsection{Rapidly increasing sequences} These sequences are used, among other things, to identify ``small'' operators in HI-type constructions. Here, as in \cite{argyros:haydon:2011}, they are used to characterize compact operators.

\begin{dfn}
Let $C > 0$. A block sequence $(x_k)_{k\in I}$ in $\X$, indexed over an interval $I$ of $\N$, will be called a $C${\em-rapidly increasing sequence} (or $C$-RIS) if there exists an increasing sequence $(j_k)_{k\in I}$ in $\N$ such that the following hold for $k\in I$:
\begin{enumerate}[label=(\roman*)]

\item $\|x_k\| \leq C$,

\item if $k>\min(I)$, $j_{k-1} < \min\supp(x_k)$, and

\item for $\ga\in\Ga$ with $\weight(\ga) = m_i^{-1} > m_{j_k}^{-1}$, $|e_\ga^*(x_k)| \leq Cm^{-1}_i$.

\end{enumerate}
Suppressing the constant $C$, $(x_k)_{k\in I}$ will be called a RIS.
\end{dfn}

\begin{exa}
\label{basis RIS}
For every  sequence $(\ga_k)_k$ in $\Ga$ with $\rank(\ga_k)\to\infty$, $(d_{\ga_k})_k$ has a subsequence that is s RIS.
\end{exa}

Indeed, for any $\ga$, $\ga'\in\Ga$ with $\weight(\ga') = m_i^{-1}\neq \weight(\ga)$, Proposition \ref{X_K evaluation analysis} yields that, unless $e_{\ga'}^*(d_\ga) = 0$, there exist $\la\in\D$ and $\eta\in\Ga$ such that $e_{\ga'}^*(d_\ga) = m_i^{-1}\la e_\eta^*(d_\ga)$, i.e., $|e_{\ga'}^*(d_\ga)|\leq 2m_i^{-1}$. If $\lim_k\weight(\ga_k) = 0$, then it is not difficult to pick a subsequence of $(d_{\ga_k})_k$ that is a $2$-RIS. Otherwise, if there is $i_0$ so that for infinitely many $k\in\N$, $\weight(\ga_k) = m_{i_0}^{-1}$, then  $(d_{\ga_k})_k$ has subsequence that is a $(2m_{i_0})$-RIS.

\begin{rem}
If $(x_k)_k$ and $(y_k)_k$ are RISs and $(\la_k)_k$, $(\mu_k)_k$ are bounded sequences of scalars, then $(\la_k x_k+\mu_k y_k)_k$ has a subsequence that is a RIS.
\end{rem}

\begin{comment}
\begin{prop}
Let $j_0\in\N$, $C>0$, and $(x_k)_{k=1}^{n_{j_0}}$ be a $C$-RIS. Then,
\begin{equation*}
\Big\|\frac{1}{n_{j_0}}\sum_{k=1}^{n_{j_0}}x_k\Big\| \leq 10Cm_{j_0}^{-1}.
\end{equation*}
\end{prop}
\end{comment}

\begin{prop}
\label{ris to zero compact}
Let $Y$ be a Banach space and $T:\X\to Y$ be a bounded linear operator. If $\lim_k\|Tx_k\| = 0$ for every RIS $(x_k)_k$ in $\X$ then $\lim_k\|Tx_k\| = 0$ for any bounded block sequence in $\X$.
\end{prop}

\begin{prop}
\label{shrinking basis}
The basis $((d_\ga)_{\ga\in\De_n})_{n=1}^\infty$ of $\X$ is shrinking and in particular $\X^*$ is isomorphic to $\ell_1$.
\end{prop}

\begin{proof}[Comment on proof]
These are \cite[Proposition 5.11 and Proposition 5.12, page 27]{argyros:haydon:2011}. The proofs are fundamental to the Argyros-Haydon construction and highly non-trivial. That being said, they translate almost verbatim to this paper. The unconvinced reader may retrace \cite[pages 19-28]{argyros:haydon:2011} (which also use \cite[Lemma 2.4, page 5]{argyros:haydon:2011}) with having only two things in mind.

In \cite{argyros:haydon:2011} the evaluation analysis of each $\gamma$ contains components of the form $b_r^*\circ P_{(s,\infty)}$, where $b_r^*$ is a convex combination of certain $e_\eta^*$. Here, in the same place these components are of the form $e_{\eta}^*\circ P_{I_r}$, where $I_r$ is a bounded interval of $\N$. This has the consequence that in several places intervals of the form $(s,\infty)$ need to be replaced with bounded intervals $I$. This does not cause any change in constants, because in \cite{argyros:haydon:2011}, $\|P_{(s,\infty)}\| \leq 3$ whereas here $\|P_I\|\leq 3$ (see Remark \ref{component norms}). This is due to the present choice $m_1 \geq 8$.

The second detail, is that in \cite{argyros:haydon:2011} there exists a unique $\gamma_0$ of weight zero, namely the member of $\De_1$. Here, there are infinitely many $\ga$ in $\Ga$ of weight zero. By their nature, these can be treated exactly as $\ga_0$ of \cite{argyros:haydon:2011}. This comes up in the proof of \cite[Proposition 5.4 (Basic Inequality), page 20]{argyros:haydon:2011}.
\end{proof}

Recall that in a space with a shrinking basis, all bounded block sequences are weakly null and thus a compact operator maps them to norm-null sequences. Therefore, the above two propositions immediately imply the next.

\begin{cor}
\label{ris compact char}
A bounded linear operator $T:\X\to\X$ is compact if and only if for every RIS $(x_k)_k$ in $\X$, $\lim_k\|Tx_k\| = 0$.
\end{cor}

\subsection{Exact pairs and dependent sequences}
These are highly specialized sequences of vectors that use RISs as their building blocks. Here, as in \cite{argyros:haydon:2011}, their main purpose is to extract the compact part of a bounded linear operator. In this section the definition of these objects is recalled and it is reminded how they can be constructed from RISs. Estimates of the norm of their linear combinations are also given.

\begin{dfn}
Let $C>0$, $j\in\N$, and $\e\in\{0,1\}$. A pair $(x,\eta)\in\X\times\Ga$ is called a {\em$(C,j,\e)$-exact pair} if
\begin{enumerate}[label=(\roman*)]

\item for all $\xi\in\Ga$, $|d_\xi^*(x)| \leq Cm^{-1}_j$,

\item $\weight(\eta)=m_j^{-1}$,

\item $\|x\|\leq C$ and $e_\eta^*(x) = \e$, and

\item for every $\eta'\in\Ga$ with $\weight(\eta') = m_i^{-1}\neq m_j^{-1}$,
\[
|e_{\eta'}^*(x)| \leq \left\{
	\begin{array}{ll}
	C m_i^{-1} & \mbox{if } m_i^{-1}>m_j^{-1},\\
	C m_j^{-1} & \mbox{if } m_i^{-1}<m_j^{-1}.
	\end{array}
\right.
\]

\end{enumerate}
\end{dfn}

The next lemma explains how to construct $(C,0)$-exact pairs in $\X$, which are necessary in the study of operators. On contrast, $(C,1)$-exact pairs are used to study the geometry of $\X$, e.g., to prove that $\X$ contains no unconditional sequences (see Section \ref{extra stuff}). The construction of $(C,0)$-exact pairs is done very similarly as in \cite{argyros:haydon:2011} and by applying Proposition \ref{gamma builder}.

A skipped block sequence $(x_k)_k$ in $\X$ is one for which $\max\supp(x_k) + 1 <\min\supp(x_{k+1})$, for all $k$.

\begin{lem}
\label{EP builder}
Let $j\in\N$, $C>0$, $(x_k)_{k=1}^{n_{2j}}$ be a skipped block $C$-RIS, $(q_k)_{k=1}^{n_{2j}}$ be a sequence of natural numbers, $(I_k)_{k=1}^{n_{2j}}$ be finite intervals of $\N$, $\eta_k\in\Ga$ and $\la_k\in\D_{q_k}$, for $k=1,\ldots,n_{2j}$, and $\kappa_0\in K$. Assume that the following are satisfied.
\begin{enumerate}[leftmargin=21pt,label=(\roman*)]

\item $2j\leq q_1$, $\supp(x_1)\cup I_1\subset[1,q_1)$, and $\supp(x_k)\cup I_k\subset (q_{k-1},q_k)$ , for $k=2,\ldots,n_{2j}$.

\item $\rank(\eta_k)<q_k$, for $k=1,\ldots,n_{2j}$,

\item $\weight(I_k,\eta_k) \leq 2^{-q_{k-1}}$, for $k=2,\ldots,n_{2j}$.

\item $\varrho(\kappa(\eta_k),\kappa_0) \leq 2^{-(q_{k-1}+1)}$, for $k=2,\ldots,n_{2j}$.

\item $e_{\eta_k}^*(P_{I_k}x_k) = 0$, for $k=1,\ldots,n_{2j}$.

\end{enumerate}
Denote
\[z = \frac{m_{2j}}{n_{2j}}\sum_{k=1}^{n_{2j}}x_k.\]

Then, there exists $\ga\in\Ga$ with $\rank(\ga) = q_{n_{2j}}$, $\weight(\ga) = m_{2j}^{-1}$, and $\varrho(\kappa(\ga),\kappa_0) \leq 2^{-q_{n_{2j}}}$ such that $(z,\ga)$ is a   $(16C,2j,0)$-exact pair. Furthermore, there exist $\xi_k\in\De_{q_k}$, for $k=1,\ldots,n_{2j}$, such that $e_\ga^*$ has an evaluation analysis
\[e_\ga^* = \sum_{k=1}^{n_{2j}}d_{\xi_k}^* + \frac{1}{m_{2j}}\sum_{k=1}^{n_{2j}}\la_ke_{\eta_k}^*\circ P_{I_k}.\]
\end{lem}

\begin{proof}
Apply Proposition \ref{gamma builder} to find $\gamma$ with the desired weight, proximity to $\kappa_0$, and evaluation analysis. The proof that $(z,\gamma)$ is a $(16C,2j,0)$-exact pair is identical to the proof of \cite[Lemma 6.2, page 29]{argyros:haydon:2011}.
\end{proof}

Dependent sequences comprise exact pairs. They are chosen inductively with the help of the coding function $\sigma$. 

\begin{dfn}
Let $C>0$, $j_0\in\N$, and  $\e\in\{0,1\}$. A sequence $(x_i)_{i=1}^{n_{2j_0}-1}$ in $\X$ is called a {\em$(C,2j_0-1,\e)$-dependent sequence} if there exists a $\ga\in\Ga$ with $\weight(\ga) = m_{2j_0}^{-1}$ and evaluation analysis
\[
e_\ga^* = \sum_{i=1}^{n_{2j_0-1}}d_{\xi_i}^* + \frac{1}{m_{2j_0-1}}\sum_{i=1}^{n_{2j_0-1}} e_{\eta_i}^*\circ P_{I_i}
\]
such that $\range(x_i)\subset I_i$, for $i=1,\ldots,n_{2j_0-1}$, denoting $\weight(\eta_1) = m_{4j_1-2}$ then $(x_1,\eta_1)$ is a $(C,4j_1-2,\e)$-exact pair, and denoting $\weight(\eta_i) = m_{4j_i}$, for $i=2,\ldots,n_{2j_0-1}$ then $(x_i,\ga_i)$ is a $(C,4j_i,\e)$-exact pair,  for $i=2,\ldots,n_{2j_0-1}$.
\end{dfn}

Note that in the construction of $\ga$, after each step $i$, ${\eta_{i+1}}$ is only allowed to have the weight $m^{-1}_{4\sigma(\xi_i)}$. Therefore, the exact pair $(x_{i+1},\eta_{i+1})$ needs to be built after $\xi_i$.

It is straightforward to check that the average of the terms of a $(C,2j_0-1,1)$-dependent sequence have norm at least $m_{2j_0-1}^{-1}$. However, for $(C,2j_0-1,0)$-dependent sequence the outcome is much smaller.

\begin{prop}
\label{zero dependent sequence bound}
Let $C>0$, $j_0\in\N$, and $(x_i)_{i=1}^{n_{2j_0}-1}$ be a $(C,2j_0-1,0)$-dependent sequence in $\X$. Then,
\[\Big\|n_{2j_0-1}^{-1}\sum_{i=1}^{n_{2j_0-1}}x_i\Big\| \leq 30Cm_{2j_0-1}^{-2}.\]
\end{prop}

\begin{proof}
This is proved identically to \cite[Proposition 6.6, pages 30-31]{argyros:haydon:2011}. 
\end{proof}

\section{Bounded linear operators on $\X$}
\label{operators section}
In this section the common HI concepts are combined with the weight and metric constraints to prove Theorem \ref{eventual continuity} and Theorem \ref{compact zero diagonal}.

\subsection{Non-vanishing estimates of very fast growing sequences}
The following states that any block sequence whose norm is bounded from below admits non-vanishing estimates by functions of the form $e^*_\ga\circ P_I$ with $\weight(I,\ga)$ tending to zero. This is necessary to be able to construct zero dependent sequences. This process has its roots in \cite{argyros:motakis:2020}.

\begin{prop}
\label{vfg estimates}
Any block sequence $(x_k)_k$ in $\X$ with $\liminf_k\|x_k\|>0$ has a subseqeuence, again denoted by $(x_k)_k$, with the following property. For each $k\in\N$ there exist an interval $I_k$ or $\range(x_k)$ and $\eta_k\in\Ga$ such that
\[\lim_k\weight(I_k,\eta_k) = 0\text{ and }\liminf_k\big|e^*_{\eta_k}\big(P_{I_k}x_k\big)\big|>0.\]
\end{prop}

The following lemma is the main quantitative argument required in the proof of the above proposition.

\begin{lem}
\label{self fulfilling}
Let $x\in\X$ and $\ga\in\Ga$ with $\weight(\ga) = m_j^{-1}$ such that $|e_\ga^*(x)|\geq(3/4)\|x\|$. Then there exist $\eta\in\Ga$ and a finite interval $I$ of $\range(x)$ with
\[\weight(I,\eta) \leq 2^{-\min\supp(x)}\text{ and }\big|e_{\eta}^*(P_I x)\big| \geq \frac{1}{8n_j}\|x\|.\]
\end{lem}

\begin{proof}
Apply Proposition \ref{X_K evaluation analysis} to write
\begin{align*}
e_\ga^* &= \sum_{r=1}^ad_{\xi_r}^* + \frac{1}{m_j}\sum_{r=1}^a\la_re_{\eta_r}^*\circ P_{I_r}.
\end{align*}
Put $J =\range(x)$, $r_0 = \min\{r:\rank(\xi_r) \geq \min(J)\}$ and for $r= r_0,\ldots,a$, $I_r' = I_r\cap J$. Recall that $a\leq n_j$. Assuming that the conclusion is false,
\begin{align*}
\frac{3}{4}\|x\| &\leq |e_\ga^*(x)| = \Big|\sum_{r=1}^ae_{\xi_r}^*(P_{\{\rank(\xi_r)\}}x) + \frac{1}{m_j}\sum_{r= r_0}^ae_{\eta_r}^*\big(P_{I'_r}x\big)\Big|\\
&\leq \frac{1}{m_j}\big|e_{\eta_{r_0}}^*\big(P_{I'_{r_0}}x\big)\big| + \sum_{r=1}^a\big|e_{\xi_r}^*(P_{\{\rank(\xi_r)\}}x)\big| + \frac{1}{m_j}\sum_{r= r_0+1}^a\big|e_{\eta_r}^*\big(P_{I'_r}x\big)\big|.
\end{align*}
By definition, for $r=r_0,\ldots,a$, $\weight(\{\rank(\xi_r)\},\xi_r) = 0$. Furthermore, the growth condition on weights in the sets $\De_n^{\text{\ref{age nonzero}}}$ yields that for $r=r_0+1,\ldots,a$, $\weight(I_r,\eta_r) \leq 2^{-\rank{\xi_{r_0}}} \leq 2^{-\min\supp(x)}$. If ones supposes that the conclusion fails, it would follow that
\[\frac{3}{4}\|x\| < \frac{4\|x\|}{m_j} + n_j\frac{1}{8n_j}\|x\| + \frac{n_j}{m_j}\frac{1}{8n_j}\|x\| \leq \Big(4\frac{1}{m_1} + \frac{1}{4}\Big)\|x\|,\]
i.e., $m_1 < 8$ which is absurd.
\end{proof}

\begin{proof}[Proof of Proposition \ref{vfg estimates}]
For each $k\in\N$, pick $\ga_k\in\Ga$ with $|e_{\ga_k}^*(x_k)|\geq (3/4)\|x_k\|$. Distinguish two cases. If $\lim_k\weight(\range(x_k),\ga_k) = 0$ then the proof is complete. Note that this also includes the case in which, for all $k\in\N$, $\min\range(x_k) = \rank(\ga_k)$ and thus $\weight(\range(x_k),\ga_k) = 0$. Otherwise, by passing to an infinite subsequence and relabeling, there exists $j_0\in\N$ so that for all $k\in\N$, $\weight(\ga_k) = m^{-1}_{j_0}$. Apply Lemma \ref{self fulfilling} to find for each $k\in\N$ a finite interval $I_k$ of $\range(x_k)$ and $\eta_k\in\Ga$ with $\weight(I_k,\eta_k)\leq 2^{-\min\supp(x_k)}$ and $|e_{\eta_k}^*(P_{I_k}x_k)| \geq 1/(8n_{j_0})$.
\end{proof}

\subsection{Linear transformations of rapidly increasing sequences}
This section deals with controlling the action of bounded linear operators on RISs. In \cite{argyros:haydon:2011} it was shown that for a bounded linear operator $T$ and a RIS $(x_k)_k$, $\mathrm{dist}(Tx_k,\C x_k)\to0$. Here, this is not true (unless $K$ is a singleton). If it were, $\X$ would have the scalar-plus-compact property. Instead, it is shown that $T$ is in a certain weak sense close to a diagonal operator. Recall that eventually it can be shown that $T$ is a compact perturbation of such an operator (Corollary \ref{diag+comp}).

\begin{prop}
\label{norm killer}
Let $T:\X\to\X$ be a linear operator, $(x_k)_k$ be a RIS in $\X$, and $(I_k,\eta_k)_k$ be a sequence, whose each term is a pair of an  interval $I_k$ of $\N$ with an $\eta_k$ in $\Ga$, such that $\lim_k\weight(I_k,\eta_k) = 0$. If
\[e_{\eta_k}^*\big(P_{I_k}x_k\big) = 0,\text{ for all }k\in\N,\text{ and }\liminf_k\big|e_{\eta_k}^*\big(P_{I_k}Tx_k\big)\big|>0,\]
then $T$ is unbounded.
\end{prop}

\begin{proof}
Let $(x_k)_k$ be a $C$-RIS and assume that $T$ is bounded. Recall that the basis of $\X$ is shrinking and therefore $(Tx_k)_k$ is weakly null.  By applying a compact perturbation to $T$ and passing to a subsequence, it may be assumed that $(Tx_k)_k$ is a block sequence and that for some $\e>0$ and each $k\in\N$, $e_{\gamma_k}^*(P_{I_k}Tx_k) = |e_{\gamma_k}^*(P_{I_k}Tx_k)| \geq \e$ (after perhaps applying a complex rotation to each $x_k$). By restricting each $I_k$ it may be assumed that it is a subset of the smallest interval $J_k$ containing $\range(x_k)\cup\range(Tx_k)$. Using the compactness of $K$, it may also be assumed that there is $\kappa_0\in K$ with $\lim_k\varrho(\kappa(\eta_k),\kappa_0) = 0$. With this information, in two steps it is possible to construct a vector that ``blows up'' the norm of $T$.

{\bf Step 1:} For each $j,q\in\N$ and $\delta>0$, there exists a $(16C,2j,0)$-exact pair $(z,\zeta)$ such that $\min\supp(z) > q$, $\varrho(\kappa(\eta),\kappa_0) \leq \delta$, and $e_\zeta^*(Tz) = |e_\zeta^*(Tz)| \geq \e$.

This is achieved with the help of Lemma \ref{EP builder}. Choose members of the sequence $x_{k_1},\ldots,x_{k_{n_2j}}$, starting after $q$, and $q_1<\cdots<q_{n_{2j}}$ such that the assumptions of that Lemma are satisfied, while at the same time $I_1\subset J_1\subset[1,q_1)$ and $I_k\subset J_k \subset (q_{r-1},q_r)$, for $r=2,\ldots,n_{2j}$. Provided that $q_{n_{2j}}$ is sufficiently large and using $\la_1=\cdots=\la_{n_{2j}} = 1$, the resulting pair $(z,\zeta)$ has the desired properties.

{\bf Step 2:} For each $j_0\in\N$ there exists a vector $w$ in $\X$ with $\|w\| \leq 480Cm_{j_0}^{-2}$ and $\|Tw\| \geq \e/m_{2j_0-1}$. In particular, $\|T\| \geq m_{j_0}\e/{480 C}$.

This vector $w$ is found by performing a similar construction process as in the proof of Proposition \ref{gamma builder} to cosntruct a $(16C,2j_0,0)$-dependent sequence $(z_i)_{i=1}^{n_{2j_0-1}}$ and a $\ga\in\Ga$ ``close'' to $\kappa_0$ with $\weight(\ga) = m_{2j_0-1}^{-1}$, $\kappa(\ga)$ and evaluation analysis
\[e_\ga^* = \sum_{i=1}^{n_{2j_0-1}}d^*_{\xi_i} + \frac{1}{m_{2j_0-1}}\sum_{i=1}^{n_{2j_0-1}}e^*_{\zeta_i}\circ P_{J_i},\]
where $J_i$ is an interval of $\N$ containing $\range(z_i)$ and $\range(Tz_i)$. In each step, if $(z_i)_{i=1}^{i_0}$ and $\xi_{i_0}$ have been chosen, use Step 1 to find a $(16C,4\sigma(\xi_{i_0}),0)$-exact pair $(z_{i_0+1},\zeta_{i_{0+1}})$ with $\min\supp(z_{i_0+1}) > \rank(\xi_{i_0})$, $\varrho(\kappa(\zeta_{i_0+1}),\kappa_0)\leq 2^{-(\rank(\xi_{i_0})+1)}$, and $e_{\eta_{i_0+1}}^*(Tz_{i_0+1}) = |e_{\eta_{i_0+1}}^*(Tz_{i_0+1})| \geq\e$. Pick $q_{i_0+1}>\max(\range(z_{i_0+1})\cup\range(Tz_{i_0+1})$ and choose an appropriate $\xi_{i_0+1}\in\De_{q_{i_0+1}}^{\text{\ref{age nonzero},odd}}$.

Then, put $w = n_{2j_0-1}^{-1}\sum_{i=1}^{n_{2j_0-1}}z_i$. It follows that $e_\ga^*(Tz) = |e_\ga^*(Tz)| \geq \e/m_{2j_0-1}$ and, by Proposition \ref{zero dependent sequence bound}, $\|z\| \leq 480Cm^{-2}_{2j_0-1}$.

Because $j_0$ was arbitrary, $\|T\|$ cannot be finite and thus $T$ is unbounded.
\end{proof}

The next result may be viewed as first step towards the fact that a bounded $T$ must be close to a diagonal operator (Corollary \ref{diag+comp}).
\begin{prop}
\label{first diagonal proximity}
Let $T:\X\to \X$ be bounded linear operator. Then,
\begin{enumerate}[leftmargin=21pt,label=(\roman*)]

\item for any RIS $(x_k)_k$ in $\X$, $\lim_k\|Tx_k - P_{\range(x_k)}Tx_k\| = 0$ and

\item $\lim_{\rank(\ga)\to\infty}\big\|d_\ga^*(Td_\ga)d_\ga - Td_\ga\big\| = 0$.

\end{enumerate}
\end{prop}

\begin{proof}
To prove (i), for each $k\in\N$ write $\range(x_k) = [p_k,q_k]$ and set $I_k = [1,p_k)$, $I'_k = (q_k,\infty)$. If (i) is false, either $\liminf_k\|P_{I_k}Tx_k\| > 0$, or $\liminf_k\|P_{I_k'}Tx_k\| > 0$. These cases are treated in the same way so assume the first one holds. It may also be assumed that $(Tx_k)_k$ is a block sequence. Therefore, $(P_{I_k}Tx_k)_k$ is also a block sequence with $\liminf_k\|P_{I_k}Tx_k\| > 0$. Apply Proposition \ref{vfg estimates} and pass to a subsequence to find, for each $k\in\N$, an interval $J_k$ of $\range(P_{I_k}Tx_k)\subset I_k$ and $\eta_k$ in $\Ga$ such that $ \lim_k\weight(J_k,\eta_k) = 0\text{ and }\liminf_k\big|e^*_{\eta_k}\big(P_{J_k}x_k\big)\big|>0$. But then, because $\range(x_k)\cap I_k = \emptyset$, it follows that $e_{\eta_k}^*(P_{J_k}x_k) = 0$, for all $k\in\N$. By Proposition \ref{norm killer}, $T$ is unbounded.

To proceed with the proof of (ii), recall that for each $n\in\N$, $(d_\ga)_{\ga\in\De_n}$ is 2-equivalent to the unit vector basis of $\ell_\infty(\De_n)$. Therefore, for each $\ga_0\in\De_n$,
\begin{align*}
\Big\|P_{\{\rank(\ga_0)\}}Td_{\ga_0} - d_{\ga_0}^*\big(Td_{\ga_0}\big)d_{\ga_0}\Big\| &= \Big\|\sum_{\ga'\in\De_n\setminus\{\ga_0\}}d_{\ga'}^*\big(Td_{\ga_0}\big)d_{\ga'}\Big\|\\
&\leq 2\max_{\ga'\in\De_n\setminus\{\ga_0\}}\big|d_{\ga'}^*\big(Td_{\ga_0}\big)\big|.
\end{align*}
Therefore, it suffices to check that
\[\lim_n\max_{\ga\neq\ga'\in\De_n}\Big|d^*_{\ga'}\big(Td_\ga\big)\Big| = 0.\]
If this is false then there exist $\e>0$, a strictly increasing sequence $(q_k)_k$ in $\N$, and for each $k\in\N$, $\ga_k\neq\ga'_k\in\De_{q_k}$, such that $|e_{\ga'_k}^*(P_{\{q_k\}}Td_{\ga_k})| = |d_{\ga_k'}(Td_{\ga_k})| \geq \e$. Observe that $\weight(\{q_k\},\ga_k') = 0$ and $e_{\ga_k'}(P_{\{q_k\}}d_{\ga_k}) = d_{\ga_k'}^*(d_{\ga_k}) = 0$, for all $k\in\N$. By passing to a subsequence, as in example \ref{basis RIS}, $(d_{\ga_k})_k$ is a RIS. Proposition \ref{norm killer} yields that $T$ is unbounded.
\end{proof}

\subsection{Eventual continuity of diagonal entries}
In this section the two black-box theorems of Section \ref{impact section} are proved. For convenience, their statements are repeated as propositions.
\begin{prop}[Theorem \ref{eventual continuity}]
\label{eventual continuity proof}
Let $T:\X\to\X$ be a bounded linear operator. Then, for every $\e>0$ there exist $n\in\N$ and $\delta>0$ such that for all $\ga$,$\ga'\in\Ga$ with $\min\{\rank(\ga),\rank(\ga')\}\geq n$ and $\varrho(\kappa(\ga),\kappa(\ga'))<\delta$,
\begin{equation}
\label{eventual continuity proof eq}\big|d_\ga^*\big(Td_\ga\big) - d_{\ga'}^*\big(Td_{\ga'}\big)\big| <\e.
\end{equation}
Therefore, the function $\varphi_T:K\to \C$ given by
\[\varphi_T(\kappa) = \lim_{\substack{\rank(\ga)\to\infty\\\kappa(\ga)\to\kappa}}d_\ga^*\big(Td_\ga\big)\]
is well defined and continuous.
\end{prop}

\begin{proof}
If the conclusion is false, there exist sequences $(\eta_k)$, $(\zeta_k)$ and $\kappa_0\in K$ such that $\varrho(\kappa(\eta_k),\kappa_0)\to 0$ and $\varrho(\kappa(\zeta_k),\kappa_0) = 0$ yet $\lim_kd_{\eta_k}^*(Td_{\eta_k}) = \la\neq\mu = \lim_kd_\zeta^*(Td_\zeta)$. By virtue of Proposition \ref{first diagonal proximity} (ii), and by perturbing $T$ by a compact operator, it may be assumed that for all $k\in\N$, $Td_{\eta_k} = \la d_{\eta_k}$ and $Td_{\zeta_k} = \mu d_{\zeta_k}$. This will make it possible to perform a process similar to that in the proof of Proposition \ref{norm killer} to ``blow up'' the norm of $T$.

Note that for some $C>0$, it is possible to pass to subsequences of $(\eta_k)$ and $(\zeta_k)$ so that if
\[(\theta_k)_k = ({\eta_1},{\zeta_1},{\eta_2},{\zeta_2},\ldots,{\eta_k},{\zeta_k},\ldots)\]
then the sequence $((-1)^kd_{\theta_k})$ is a skipped block $C$-RIS. This argument is similar to Example \ref{basis RIS}. If $\lim_k\weight(\eta_k) = \lim_k\weight(\zeta_k) = 0$ then one can choose this sequence to be a $2$-RIS. If instead there exists $j_0\in\N$ such that $\limsup_k\max(\weight(\eta_k),\weight(\zeta_k)) = m_{j_0}^{-1}$, then $((-1)^kd_{\theta_k})$ may be chosen to be $2m_{j_0}$-RIS. Furthermore, while choosing $(\theta_k)_k$, it is possible to choose a sequence of natural numbers $(q_k)_k$ with $\rank(\theta_1)<q_1<\rank(\theta_2)<q_2<\cdots$ such that, for each $k\geq 2$, $\varrho(\kappa(\theta_k),\kappa_0)\leq 2^{-(q_{k-1}+1)}$. Also recall that for each $k\in\N$, $d_{\theta_k}^*  = e_{\theta_k}^*\circ P_{\{\rank(\theta_k)\}}$ and $\weight(\{\rank(\theta_k)\},\theta_k) = 0$.

{\bf Step 1:} For each $j,q\in\N$ and $\delta>0$, there exists a $(16C,2j,0)$-exact pair $(z,\zeta)$ such that $\min\supp(z) > q$, $\varrho(\kappa(\eta),\kappa_0) \leq \delta$, and $|e_\zeta^*(Tz)| \geq \e = |\la - \mu|/2$.

This is a repetition of the argument in Lemma \ref{EP builder}. Omitting a few initial terms of the sequence $((-1)^kd_{\theta_k})_k$, there exists a $\zeta\in\Ga$ with evaluation analysis
\[e_\zeta^* = \sum_{k=1}^{n_{2j}}d_{\xi_k}^* + \frac{1}{m_{2j}}\sum_{k=1}^{n_{2j}}d^*_{\theta_k},\]
where $\rank(\xi_k) = q_k$, for $k=1,\ldots,n_{2j}$. Put $z = m_{2j}n_{2j}^{-1}\sum_{k=1}^{n_{2j}}(-1)^kd_{\theta_k}$. Then, $(z,\zeta)$ is a  $(16C,2j,0)$-exact pair (recall that $n_{2j}$ is even) and $e^*_\zeta( Tz) = (\la-\mu)/2$.

{\bf Step 2:} For each $j_0\in\N$ there exists a vector $w$ in $\X$ with $\|w\| \leq 480Cm_{j_0}^{-2}$ and $\|Tw\| \geq \e/m_{2j_0-1}$. In particular, $\|T\| \geq m_{j_0}\e/{480 C}$.

This is identical to the proof of the second step in Proposition \ref{norm killer} and, in conclusion, $T$ is unbounded.
\end{proof}

\begin{prop}[Theorem \ref{compact zero diagonal}]
\label{compact zero diagonal proof}
A bounded linear operator $T:\mathcal{L}(\X)\to\mathcal{L}(\X)$ is compact if and only if $\varphi_T = 0$.
\end{prop}

\begin{proof}
If $T$ is compact, then it sends weakly null sequences to norm-null ones. Because $(d_\ga)_{\ga\in\Ga}$ is shrinking, it follows that $\varphi_T = 0$.

Assume now that $T$ is a bounded linear operator with $\varphi_T = 0$, i.e., $\lim_{\rank(\ga)\to\infty}d_\ga^*(Td_\ga) = 0$. Proposition \ref{first diagonal proximity} (ii) yields
\begin{equation}
\label{compact zero diagonal proof eq1}
\lim_{\rank(\ga)\to\infty}\big\|Td_\ga\big\| = 0.
\end{equation}
Towards contradiction, assume that $T$ is not compact. Then, there exists a RIS $(x_k)_k$ with $\limsup_k\|Tx_k\| > 0$. As usual, assume that $(Tx_k)_k$ is a block sequence for. By Proposition \ref{vfg estimates} one may pass to a further subsequence and find $\e > 0$ such that, for each $k\in\N$, there exist an interval $I_k$ of $\range(Tx_k)$ and $\eta_k\in\Ga$ such that $|e_{\eta_k}^*(P_{I_k}Tx_k)| \geq \e$ and $\lim_k\weight(I_k,\eta_k) = 0$.

Assume for the moment that $\rank(\eta_k)\in I_k$, for all $k\in\N$. Although this might not be true, later the general case will be reduced to this one. For each $k\in\N$, consider the sequence $y_k = x_k - e_{\eta_k}^*(P_{I_k}x_k)d_{\eta_k}$. By passing to a subsequence, this is a RIS and because it was assumed that $\rank(\eta_k)\in I_k$, for all $k\in\N$, it follows that
\[e_{\eta_k}^*\big(P_{I_k}y_k\big) = e_{\eta_k}^*\big(P_{I_k}x_k\big) - e_{\\eta_k}^*\big(P_{I_k}x_k\big)e_{\eta_k}^*\big(P_{I_k}d_{\eta_k}\big) = 0.\]
By \eqref{compact zero diagonal proof eq1}, $\liminf|e_{\eta_k}^*(P_{I_k}Ty_k)| = \liminf|e_{\eta_k}^*(P_{I_k}Tx_k)| \geq \e$, By Proposition \ref{norm killer}, $T$ is unbounded.

To treat the remaining case, assume that $\rank(\eta_k)\notin I_k$, for all $k\in\N$. This in particular implies that $\rank(\eta_k)>\max(I_k)$ because otherwise $e_{\eta_k}^*\circ P_{I_k} = 0$. In particular, $\weight(I_k,\eta_k) = \weight(\eta_k)$, for all $k\in\N$. Let $I_k'$ be the smallest interval of $\N$ containing $I_k$ and $\rank(\eta_k)$. Write $I_k' = I_k\cup I_k''$, where $I_k'' = \{\max(I_k) +1,\ldots,\rank(\eta_k)\}$. Then, $e_{\eta_k}^*\circ P_{I_k} = e_{\eta_k}^*\circ P_{I'_k} - e_{\eta_k}^*\circ P_{I_k''}$. Therefore,
\[\text{either }\big|e_{\eta_k}^*\big(P_{I'_k}x_k\big)\big| \geq\e/2\text{ or }\big|e_{\eta_k}^*\big(P_{I'_k}x_k\big)\big| \geq\e/2.\]
Let $J_k = I_k'$ or $J_k = I_k''$ accordingly. In either case, $\rank(\eta_k)\in J_k$ and $\weight(J_k,\eta_k) \leq \weight(\eta_k) = \weight(I_k,\eta_k)\to 0$. Note that it might no longer be true that $J_k\subset\range(x_k)$, but this was not required in the proof of the special case above.
\end{proof}

\section{Additional properties of $\X$}
\label{extra stuff}
This section outlines some additional properties of $\X$ without giving detailed proofs. For example, the complemented subspaces of $\X$ are classified and it is mentioned that $\X$ does not contain unconditional basic sequences.

\subsection{Complemented subspaces of $\X$}
Denote by $\mathrm{Clopen}(K)$ the collection of clopen subsets of $K$. The wealth of projections on $\X$ is directly linked to this collection. For each $F\in\mathrm{Clopen}(K)$, the characteristic $\chi_F:K\to\C$ is Lipschitz and therefore the diagonal operator $\hat\chi_K:\X\to\X$ is a bounded 0-1 valued diagonal operator, in particular it is the canonical projection onto the subspace $W_F = \overline{\langle\{d_\ga:\kappa(\ga)\in F\}\rangle}$. Observe that, because $F$ is open, whenever $F\neq\emptyset$ then $W_F$ is infinite dimensional. This is because the sequence $(K_n)_n$ is increasing with dense union and for each $n\in\N$, $\{\kappa(\ga):\ga\in\De_n\} = K_n$. For each $K\in\mathrm{Clopen}(K)$, denote $P_F = \hat\chi_K$.

\begin{prop}
\label{projections in X_K}
The map $F\mapsto P_F$ is an injection from $\mathrm{Clopen}(K)$ into the collection of bounded canonical basis projections such that $P_F$ is of infinite rank if and only if $F\neq\emptyset$. Furthermore the following hold.
\begin{enumerate}[leftmargin=23pt,label=(\roman*)]

\item\label{P_F+compact} For every projection $P:\X\to\X$ there exists a unique $F\in\mathrm{Clopen}(K)$ such that $P - P_F$ is compact.

%\item Every complemented subspace $W$ of $\X$ has a further subspace $Z$ with $\mathrm{dim}(W/Z) < \infty$ that is $5$-complemented in $\X$.

\item\label{Z_F mod finite} For every complemented subspace $W$ of $\X$ there exist $F\in\mathrm{Clopen}(K)$ and $n\in\N\cup\{0\}$ such that either $W\simeq W_F\oplus \C^n$ or $W\simeq W_F/\C^n$.

\item\label{Z_F diag+comp} For each $F\in\mathrm{Clopen}(K)$, $W_F$ has the diagonal-plus-compact property.

\end{enumerate}
\end{prop}

\begin{proof}[Sketch of proof]
For \ref{P_F+compact}, note that $\phi = \Psi(P)$ is an idempotent in $C(K)$ and therefore the characteristic of a clopen set $F$. By Theorem \ref{compact zero diagonal}, $P - P_F$ is compact. The uniqueness comes from the fact that for $F\neq G$, $P_F-P_G$ is non-compact.

For \ref{Z_F mod finite}, let $P$ be a projection onto $W$ and write $P = P_F + A$, with $A$ compact. By Proposition \ref{lip diag bounded}, there exists $N\in\N$ such that for all $n\geq N$, $\|P_FP_{[n,\infty)}\|\leq 4$. For $0<\e<1$ and $n$ sufficiently large, $\|AP_{[n,\infty)}\| \leq \e$. Some classical functional analysis gymnastics yield that the space $Z = P(P_FP_{[n,\infty)}\X)$ is $4(1+\e)/(1-\e)$-complemented in $\X$. Also, $Z\simeq P_FP_{[n,\infty)}\X$ and $\dim(W/Z) < n$.

Item \ref{Z_F diag+comp} follows from extending an operator on $W_F$ to $\X$ and using Corollary \ref{diag+comp}. Note that in particular the Calkin algebra of $W_F$ is $C(F)$.
\end{proof}

The next statement discusses the decomposability of $\X$.
\begin{prop}~
\begin{enumerate}[leftmargin=20pt,label=(\roman*)]

\item\label{connected-decomposable} The space $\X$ is indecomposable if and only if $K$ is connected.

\item\label{complementably decomposable} If $K$ is totally disconnected, then every infinite dimensional complemented subspace of $\X$ is decomposable.

\end{enumerate}
\end{prop}

\begin{proof}[Sketch of proof]
The first statement follows directly from Proposition \ref{projections in X_K} \ref{P_F+compact}. For the second one, by Proposition \ref{projections in X_K} \ref{Z_F mod finite} it suffices to check for a subspace of the form $W_F$. Write $F$ as the disjoint union of two non-empty clopen sets $F_1$, $F_2$ and observe that $W_F = W_{F_1} \oplus W_{F_2}$. It is perhaps of some interest that this process can be combined with Proposition \ref{lip diag bounded} to construct an uncomplemented subspace of $W_F$ with an infinite dimensional Schauder decomposition. Note that all that was required is that every non-empty clopen subset of $K$ contains a further proper non-empty clopen subset. This is strictly weaker than $K$ being totally disconnected.
\end{proof}

\begin{rem}
In all cases, $\X$ has, up to isomorphism, countably many complemented subspaces. This is because $\mathrm{Clopen}(K)$ is countable and Proposition \ref{projections in X_K} \ref{Z_F mod finite}. An argument similar to that of Proposition \ref{projections in X_K} \ref{Z_F subspace isom} shows that for $F\neq G\in\mathrm{Clopen}(K)$, $W_F$ is not isomorphic to a finite codimensional subspace of $W_G$. This can be used to resolve  the isomorphic containment relation on the complemented subspaces of $\X$ based on the inclusion relation on $\mathrm{Clopen}(K)$.
\end{rem}

\subsection{The subspace structure of $\X$}
This section discusses the lack of unconditional sequences in $\X$ and whether $\X$ is HI.

\begin{prop}~
\label{unconditionality in X_K}
\begin{enumerate}[leftmargin=23pt,label=(\roman*)]

\item\label{hi saturated} The space $\X$ does not contain unconditional sequences and therefore it is HI-saturated

\item\label{hi singleton} The space $\X$ is hereditarily indecomposable if and only if $K$ is a singleton.

\item\label{Z_F subspace isom} A complemented subspace of $\X$ is not isomorphic to any of its proper subspaces.

\item\label{infinite dim SD} The space $\X$ does not admit an infinite dimensional Schauder decomposition.

\end{enumerate}
\end{prop}

\begin{proof}[Sketch of proof]
For the first two statements, it is necessary to show that every block subspace of $\X$ contains a 2-RIS. The main difference to how this is proved in \cite{argyros:haydon:2011} is here it is necessary to use Proposition \ref{vfg estimates} and then perform an argument similar to that in \cite[Lemma 8.2]{argyros:haydon:2011}.

To show that $\X$ contains no unconditional sequence, on any block subspace one builds a RIS $(x_k)_k$ that is normalized by a very fast growing sequence $(e_{\eta_k}^*\circ P_{I_k})_k$ with $\kappa(\eta_k)\to\kappa_0$ and to build a one-exact pair. The rest of the process is similar to \cite[Section 8]{argyros:haydon:2011}. By Gowers' famous dichotomy theorem from \cite{gowers:1996}, $\X$ is HI-saturated.

If $K$ is a singleton then the metric constraint trivializes; for every $\ga,\ga'\in\Ga$, $\kappa(\ga) = \kappa(\ga')$. Therefore exact pairs coming from different block subspaces can be connected to show the HI property, as in \cite[Section 8]{argyros:haydon:2011}. If $K$ is not a singleton, then take two non-empty open sets $U$, $V$ with positive distance. Using the metric, define a Lipschitz function $\phi$ with $\phi|_U = 1$ and $\phi_V = 0$. Then, $\hat \phi_U$ restricted on $\overline{\langle\{d_\ga:\kappa(\ga)\in U\cup V\}\rangle}$ is a projection onto $\overline{\langle\{d_\ga:\kappa(\ga)\in U\}\rangle}$ with kernel $\overline{\langle\{d_\ga:\kappa(\ga)\in V\}\rangle}$.

Statement \ref{Z_F subspace isom} is an application of the Fredhold index on spaces of the form $W_F$. Note that every semi-Fredhold diagonal operator on $W_F$ is Fredholm of index zero. By Proposition \ref{projections in X_K} \ref{Z_F diag+comp}, every semi-Fredholm operator on $W_F$ is Fredholm of index zero.

The final statement follows from Proposition \ref{projections in X_K} \ref{P_F+compact} and compactness. Assume that $(P_n)_n$ is a sequence of infinite dimensional projections with increasing ranges $(W_n)_n$, such that $\dim(W_{n+1}/W_n) = \infty$ for all $n\in\N$, that converges in the strong operator topology to the  identity. For each $n\in\N$ Write $P_n = P_{F_n} + A_n$. It follows that $F_1\subsetneq F_2\subsetneq\cdots$ and $\cup_n F_n = K$. Compactness prohibits this.
\end{proof}

\begin{rem}
Although it is likely that no subspace of $\X$ is isomorphic to its proper subspaces, the lack of convex combinations in the definition of $\Ga$ does not allow the usual proofs to go through. 
\end{rem}

\subsection{The space of bounded linear operators $\mathcal{L}(\X)$}
This brief section visits the space $\mathcal{L}(\X)$ with regards to its ideal and subspace structure.
\begin{prop}~
\label{L(X) prop}
\begin{enumerate}[leftmargin=23pt,label=(\roman*)]

\item\label{c0 in L(X)} The space $\mathcal{L}(\X)$ does not contain an isomorphic copy of $c_0$. In particular, unless $K$ is finite, $\mathcal{L}(\X)$ does not contain an isomorphic copy of $C(K)$ and therefore $\mathcal{K}(\X)$ is not complemented in $\mathcal{L}(\X)$.

\item\label{quotient ss} The quotient map $[\cdot]:\mathcal{L}(\X)\to\mathpzc{Cal}(\X)$ is strictly singular if an only if $K$ is countable.

\item\label{ideals} Denote by $\mathrm{Open}(K)$ the collection of open subsets of $\X$. There exists an order preserving bijection between $\mathrm{Open}(K)$ and the non-zero closed two sided ideals of $\mathcal{L}(\X)$.

\end{enumerate}
\end{prop}

\begin{proof}[Comment on Proof]
The proof of item \ref{c0 in L(X)} is outlined in \cite[Remark 4.5, page 65]{motakis:puglisi:zisimopoulou:2016}. The fact that for $K$ countable the quotient map is strictly singular is also explained in Remark 4.6 of that same paper. If $K$ is uncountable, then $C(K)$ contains an isomorphic copy of $\ell_1$ that can be lifted by $[\cdot]$ to $\mathcal{L}(\X)$. Item \ref{ideals} was explained in \cite[Remark 1.5 (vi), page 1022]{kania:laustsen:2017}.
\end{proof}

\subsection{Very incomparable spaces with $C(K)$ Calkin algebras} Similarly to \cite[Section 10.2, page 46]{argyros:haydon:2011}, by varying the the sequence $(m_j,n_j)_j$ is is possible to create different versions of the space $\X$ that are very incomparable to one another. Recall that, up to homeomorhism, there are continuum many compact metrizable spaces, as they may be identified with the closed subsets of $[0,1]^\N$.

\begin{prop}
\label{very incomparable}
There exists a collection of Banach spaces
\[\big\{\X^\alpha: \alpha<\mathfrak{c},\;K\in\mathrm{Closed}([0,1]^\N)\big\}\]
with the following properties.

\begin{enumerate}[label=(\roman*),leftmargin=19pt]

\item For each $\X^\alpha$, $\mathpzc{Cal}(\X^\alpha)$ is homomorphically isometric to $C(K)$.

\item For $(\alpha,K)\neq (\beta,L)$ every bounded linear operator $T:\X^\alpha\to \mathfrak{X}_{^{C\>\!\!(\>\!\!L\>\!\!)}}^{\beta}$ is compact.

\end{enumerate}
\end{prop}

\begin{proof}[Comment on Proof]
This is achieved by choosing an almost disjoint family of infinite subsets of the natural numbers $\{L^\alpha_K:\alpha<\mathfrak{c},\;K\in\mathrm{Closed}([0,1]^\N)\}$. Then for each such $L^\alpha_K$ define a space $\X^\alpha$ with Calkin algebra $C(K)$ using $(m_j,n_j)_{j\in L^\alpha_K}$. The proof that for each $(\alpha,K)\neq (\beta,L)$ every bounded linear operator $T:\X^\alpha\to \mathfrak{X}_{^{C\>\!\!(\>\!\!L\>\!\!)}}^{\beta}$ is compact is very similar to \cite[Theorem 10.4, page 47]{argyros:haydon:2011}, with the assistance of Proposition \ref{vfg estimates}.
\end{proof}

\begin{cor}
\label{matrix-calkin}
Let $k_1,\ldots,k_n\in\N$ and $K_1,\ldots,K_n$ be compact metric spaces. There exists a Banach space $X$ with $\mathpzc{Cal}(X)$ isomorphic as a Banach algebra to $M_{k_1}(C(K_1))\oplus\cdots\oplus M_{k_n}(C(K_n))$.
\end{cor}

\begin{proof}[Comment on Proof]
It was observed by Kania and Laustsen in \cite[Note added in proof, page 1022]{kania:laustsen:2017} that every finite dimensional semi-simple complex algebra is a Calkin algebra. Using Proposition \ref{very incomparable}, their argument works perfectly well to find Calkin algebras of the above form.
\end{proof}

\section{Open problems}
\label{open problems}
With regards to the question of what unital algebras can be realized as Calkin algebras of a Banach space (see, e.g., Tarbard's PhD thesis \cite[page 134]{tarbard:2013} in 2012), there is a lot of progress to be made. For example, there does not exist a known property of unital Banach algebras that precludes them from being Calkin algebras (some progress was made in \cite{horvath:kania:2021} which is discussed further below). The appearance of the Argyros-Haydon construction led to the description of a plethora of explicit Calkin algebras, both finite and infinite dimensional. Unlike the Calkin algebras of classical spaces, these examples are always separable and (in essence) commutative. Additionally, all infinite dimensional ones are non-reflexive.

From a Banach-space-theoretic perspective, the construction of a reflexive Calkin algebra seems particularly intriguing and challenging. In \cite{motakis:puglisi:tolias:2020} a quasireflexive $\mathpzc{Cal}(X)$ was found. Although this may seem very close, the underlying space $X$ has an infinite dimensional Schauder decomposition. This must be avoided in order to achieve reflexivity of the Calkin algebra (compare this to Proposition \ref{unconditionality in X_K} \ref{infinite dim SD}).

\begin{problem}
Does there exist a reflexive and infinite dimensional Calkin algebra?
\end{problem}
It is worth mentioning that, if viewed simply as Banach spaces, there exists a large variety of Calkin algebras. In \cite{motakis:puglisi:tolias:2020} the author, Puglisi, and Tolias proved that there exist HI Calkin algebras and that every non-reflexive space with an unconditional basis is a Calkin algebra, albeit with a non-standard multiplication.

Constructing non-separable explicit Calkin algebras would require the development of additional tools. Of particular relevance is the question of the existence of a space with an unconditional basis that has the diagonal-plus-compact property. Such a space's Calkin algebra would automatically be isomorphic to $C(\beta\N\!\setminus\!\N)$. Another path worth exploring is the existence of some type of non-separable Argyros-Haydon space (e.g., based on the Bourgain-Pisier construction from \cite{bourgain:pisier:1983} and its non-separable version by Lopez-Abad from \cite{lopez-abad:2013}). It is worth pointing out that there already exist known examples of non-separable $C(K)$ algebras with representations of the form $\mathcal{L}(X)/\mathcal{SS}(X)$ (Koszmider, \cite{koszmider:2004} and Plebanek, \cite{plebanek:2004}). It unclear however how the methods from these two papers could be used to study the following.

\begin{problem}
Does there exist a non-separable $C(K)$ space that is a Calkin algebra?
\end{problem}

Horv\'ath and Kania proved in \cite{horvath:kania:2021} that for any cardinal $\la$ there exists a $C(K)$ space of density $2^\la$ that is not the Calkin algebra of any space of density $\la$. Of course, this does not mean that $C(K)$ is not the Calkin algebra of a space with larger density.

The current paper yields that every separable and commutative $C^*$-algebra can be represented as a Calkin algebra. Outside this class, there still remain classical commutative Banach algebras for which the existence of such a representation is unknown, e.g., the convolution algebra $L_1(G)$ for an abelian locally compact polish group $G$. Note that Tarbard's Calkin algebra from \cite{tarbard:2013} is not of this type, because it is $\ell_1(\N_0)$, a semigroup algebra. This space is closely related to the disk algebra, another example of interest. In a personal communication with the author, J. Pachl asked the following.

\begin{problem}
What semigroup algebras admit representations as Calkin algebras?
\end{problem}

There exist non-commutative explicit Calkin algebras, such as all finite dimensional semi-simple complex algebras (as observed by Laustsen and Kania in \cite[Note added in proof, page 1022]{kania:laustsen:2017}) and algebras of the form $M_{k_1}(C(K_1))\oplus\cdots\oplus M_{k_n}(C(K_n))$ (see Corollary \ref{matrix-calkin}). However, these examples are built by applying elementary processes to commutative ones.

There are additional challenges associated to describing explicit (and ``genuinely'') non-commutative Calkin algebras. In \cite{gowers:maurey:1997} Gower and Maurey gave an example of a quotient algebra of some $\mathcal{L}(X)$ that resembles the Cuntz algebra $\mathcal{O}_n$. However, in that construction it is unclear what the kernel of corresponding homomorphism is.

The direction of focusing on the description of explicit non-commutative $C^*$-algebras as Calkin algebras was proposed to the author in a personal communication by N. C. Phillips, who specifically asked the following.
\begin{problem}
Do the following non-commutative $C^*$-algebras admit representations as Calkin algebras?
\begin{enumerate}[label=(\alph*)]

\item\label{UHF} The UHF algebra of type $2^\infty$.

\item\label{Cuntz} The Cuntz algebra $\mathcal{O}_n$.

\item\label{Reduced} The reduced $C^*$-algebra of the free group on two generators, $C_r^*(\mathbb{F}_2)$.

\item\label{Full} The full $C^*$-algebra of the free group on two generators, $C^*(\mathbb{F}_2)$.

\end{enumerate}
\end{problem}

Let $\mathcal{A}$ denote a specific unital $C^*$-algebra, e.g., one of the above. The first step towards representing it as a Calkin algebra is to identify the right class of operators $\mathcal{C}$ acting on a separable Hilbert space that generates $\mathcal{A}$. The next logical step it to represent $\mathcal{A}$ as a quotient algebra of an $\mathcal{L}(X)$ space \`a-la Gowers-Maurey. This is achieved by creating a space $X$ where a class modeled on $\mathcal{C}$ reigns supreme in $\mathcal{L}(X)$, in the sense that it can be used to approximate all operators, modulo perhaps some small ideal (e.g., the strictly singular or compact operator ideal). The first hurdle in achieving this task is that the classical Gowers-Maurey HI space, being based on Schlumprecht space, resembles $\ell_1$ and not a Hilbert space. This is the precise reason why in the Gowers-Maurey shift space $X$ from \cite{gowers:maurey:1997}, $\mathcal{L}(X)$ has $\ell_1(\Z)$ as a quotient algebra instead of $C(\mathbb{T})$. The explanation is the following. Considers the class $\mathcal{C}$ of all integer powers of the right shift operator on $\Z$. Acting on $\ell_1(\Z)$ this class generates the convolution  algebra $\ell_1(\Z)$ whereas acting on $\ell_2(\Z)$ it generates $C(\mathbb{T})$.

A reasonable approach would be to first focus on Banach algebras of operators on $\ell_1$ that are similar to \ref{UHF}, \ref{Cuntz}, \ref{Reduced}, \ref{Full} e.g.,
\begin{enumerate}

\item spatial $L_p$ UHF algebras (Phillips, \cite{phillips:2013}),

\item $L_p$-Cuntz algebras $\mathcal{O}_n^p$ (Phillips, \cite{phillips:2012}),

\item reduced group $L_p$-operator algebras $F_\la^p(G)$ (Herz, \cite{herz:1973}), and

\item full group $L_p$-operator algebras $F^p(G)$ (Phillips, \cite{phillips:2013:crossed}).

\end{enumerate}
After some progress has been made in the case $p=1$, there is some available technology that may be used to generalize, namely the asymptotic-$\ell_p$ HI spaces of Deliyanni and Manoussakis from \cite{deliyanni:manoussakis:2007}. It is not entirely clear how one would then proceed to the next step, i.e., representing $\mathcal{A}$ as a Calkin algebra, but making it this far would most certainly provide a lot of insight. Calkin algebras not based on $\ell_1$ were achieved by the author, Puglisi, and Tolias in \cite{motakis:puglisi:tolias:2020} by combining techniques of Argyros, Deliyanni, and Tolias from \cite{argyros:deliyanni:tolias:2011} and Zisimopoulou from \cite{zisimopoulou:2014}.

$K$-theory of operator spaces on Banach spaces has been studied since the 1990s when Gowers and Maurey used it in \cite{gowers:maurey:1997} to prove the existence of a Banach space isomorphic to its cube, but not its square. This highlighted the connections between $K$-theory, Fredholm theory, and quotients of operator algebras in general Banach spaces. Since then, the $K$-theory of $\mathcal{L}(X)$ has been computed for various $X$  and examples of spaces with interesting $K$-theories have been constructed (see, e.g., \cite{laustsen:1999}, \cite{laustsen:2001}, \cite{zsak:2002}, and \cite{kania:koszmider:laustsen:2015}). The Gowers-Maurey and Argyros-Haydon spaces have been an important component of this endeavour. Phillips asked the author the following question, which has also been attributed to Laustsen (see, e.g., \cite[page 748]{zsak:2002}).

\begin{problem}
\label{laustsen problem}
Which pairs of abelian groups $(G_0,G_1)$ can arise as $K$-groups $(K_0(\mathcal{L}(X)),K_1(\mathcal{L}(X)))$ for some Banach space $X$?
\end{problem}

It is known among K-theory experts that for every pair of countable abelian groups $(G_0,G_1)$, there exists a compact metric space $X$ such that $(K_0(X),K_1(X)) = (\Z\oplus G_0,G_1)$. Despite the author's best effort, a reference for this general statement could not be found. An outline of a construction of such $X$ for finitely generated abelian groups can be found in, e.g., \cite[Exercise 13.2, page 228]{rordam:larsen:laustsen:2000}. The author and Phillips showed in \cite{motakis:phillips:2023}  that for every compact metric space $X$, if $E$ is the Banach space given by Theorem \ref{main theorem} with Calkin algebra $C(X)$, then $(K_0(\mathcal{L}(E)),K_1(\mathcal{L}(E))) = (\Z\oplus K_0(X),K_1(X))$.
Therefore in Problem \ref{laustsen problem}, for any pair of countable abelian groups $(G_0,G_1)$, the pair $(\Z\oplus\Z\oplus G_0,G_1)$ is realizable.
 A related question is the following.
 
\begin{problem}
\label{ref problem}
Which pairs of abelian groups $(G_0,G_1)$ can arise as $K$-groups $(K_0(\mathpzc{Cal}(X)),K_1(\mathpzc{Cal}(X)))$ for some Banach space $X$?
\end{problem}

%By Theorem \ref{main theorem}, for any pair of countable abelian groups $(G_0,G_1)$ the pair $ (\Z\oplus G_0,G_1)$ is realizable.

\section*{Acknowledgements} The author would like to thank the anonymous referee for helpful suggestions and, in particular, for recommending the inclusion of Problem \ref{ref problem} and a discussion on $K$-theory in Section \ref{open problems}.

%\appendix
%\section{}\label{appendix}
%
%
\bibliographystyle{plain}
\bibliography{bibliography}
% \nocite{*}

\end{document}